\documentclass[a4paper]{amsart}

\usepackage{graphicx,stix}
\usepackage{amsmath}
\usepackage{amsthm}
\usepackage{amssymb}
\usepackage{enumitem}
\usepackage[all,2cell]{xy}
\usepackage[
pdfstartview=FitH, unicode=true, CJKbookmarks=true, bookmarksnumbered=true,
bookmarksopen=true, colorlinks=true, citecolor=green, filecolor=cyan]{hyperref}

\author{Chao Sun}
\date{\today}
\address{School of Mathematical Sciences \\
         University of Science and Technology of China \\
        Hefei, Anhui 230026 \\
        P. R. China}
\email{maizisc@mail.ustc.edu.cn}

\UseAllTwocells
\SilentMatrices
\numberwithin{equation}{section}

\setlist{leftmargin=*,label=\emph{(\arabic*)}}
\setlist{nolistsep}

\newenvironment{enumi}
{\begin{enumerate}[itemsep=0pt,  parsep=0pt, topsep=0pt, label={$\mathrm{(\roman*)}$}, font=\normalfont,leftmargin=1cm]}
{\end{enumerate}}
\newenvironment{enum2}{
	\begin{enumerate}[itemsep=0pt,  parsep=0pt, topsep=0pt,label={$\textup{(\arabic*)}$},leftmargin=1.5pc]
}{\end{enumerate}}

\theoremstyle{plain}
\newtheorem{thm}{Theorem}[section]
\newtheorem{prop}[thm]{Proposition}

\newtheorem{cor}[thm]{Corollary}
\newtheorem{lem}[thm]{Lemma}

\theoremstyle{definition}
\newtheorem{defn}[thm]{Definition}
\newtheorem{exm}[thm]{Example}
\theoremstyle{remark}
\newtheorem{rmk}[thm]{Remark}

\def\ZZ{\mathbb{Z}}

\def\C{\mathcal{C}}

\def\D{\mathcal{D}}
\def\A{\mathcal{A}}
\def\B{\mathcal{B}}
\def\CC{\mathbb{C}}
\def\T{\mathcal{T}}

\def\I{\mathcal{I}}
\def\K{\mathcal{K}}
\def\NN{\mathbb{N}}
\def\SSS{\mathcal{S}}

\def\O{\mathcal{O}}
\def\X{\mathcal{X}}
\def\Y{\mathcal{Y}}
\def\XX{\mathbb{X}}

\def\PP{\mathbb{P}}
\def\SS{\mathbb{S}}
\def\E{\mathcal{E}}

\def\ra{\rightarrow}
\def\lra{\longrightarrow}
\def\epic{\twoheadrightarrow }

\def\monic{\hookrightarrow}

\def\GP{\textup{GProj}}
\def\Ex{\textup{Ex}}

\def\pr{\textup{pr}}

\def\coker{\textup{coker}\,}
\def\add{\textup{add\,}}

\def\hom{\textup{Hom}}
\def\End{\textup{End}}

\def\ext{\textup{Ext}}
\def\res{\textup{Res}}

\def\coh{\textup{coh}}

\def\mod{\textup{mod}\,}
\def\Mod{\textup{Mod}\,}

\def\Proj{\textup{Proj}\,}
\def\Inj{\textup{Inj}\,}
\def\proj{\textup{proj}\,}
\def\aut{\textup{Aut}}
\def\gldim{\textup{gl.dim}\,}

\def\pair#1{\langle #1 \rangle}

\def\Aut{\textup{Aut}}
\def\im{\textup{im}\,}
\def\id{\textup{Id}}

\def\op{\textup{op}}

\def\irr{\textup{Irr}}
\def\ind{\textup{Ind}}

\def\MCM{\textup{Gproj}}
 \def\gr{\textup{gr}\,}
 \def\sg{\textup{Sg}}
 \def\modZ{\textup{mod}^\ZZ}
 \def\projZ{\textup{proj}^\ZZ}

\usepackage{lineno}

\begin{document}
\title{A note on equivariantization of additive categories and triangulated categories}
\maketitle
\begin{abstract}
In this article, we investigate the category $\A^G$ of equivariant objects of an additive category $\A$ with respect to an action of a finite group $G$. We show that if $G$ is solvable then we can reconstruct  $\A$ from $\A^G$ via a finite sequence of equivariantization. We also consider the possibility of a triangulated structure on $\A^G$ canonical in certain sense  when $\A$ is triangulated and give several instances in which  $\A^G$ is indeed canonically triangulated. 

\end{abstract}
\section{Introduction}

In this article, we investigate the  category $\A^G$  of equivariant objects of an additive category $\A$ with respect to an  action of a finite group $G$. We refer the reader to \cite{deligne, DGNO} for generalities (see also \S\ref{sec: equiv})  and to \cite{BL, RR} for classical results in some classical settings. Besides these works, \cite{chen, CCZ, E, E3} are among recent works whose settings are close and/or closely related to ours. 


The motivating example of this article is as follows. 
Recall that given a collection $\underline{p}=(p_1,\dots, p_t)$ of integers  $\geq 2$ and a collection $\underline{\lambda}=(\lambda_1,\dots, \lambda_t)$ of distinct closed points of the complex projective line $\PP^1_\CC$, if $(p_1,\dots,p_t)\neq (q,q')$, where $q\neq q'$ then  there is a ramified  Galois covering $\pi\colon Y\ra X,$ where $Y$ is a smooth projective curve, such that the branch locus is  $\sum_{i=1}^t\lambda_i$ and the ramification index over $\lambda_i$ is $p_i$. Denote by $G$ the Galois group of $\pi$. Let $\XX$ be the weighted projective line attached to the data $(\underline{p}, \underline{\lambda})$ in the sense of Geigle and Lenzing~\cite{GL}. Then  we have an equivalence $\coh^G(Y)\simeq \coh\XX$ (see~\cite{Poli}). See also \cite{CCZ} for the case of a weighted projective line of virtual genus $1$.  

In the remaining of this introduction, we describe the content of this article.  For simplicity and clarity, we assume in this introduction  that all categories are linear over a field $k$ and  $|G|$ is invertible in $k$.

In \S\ref{sec: comonad}, we  briefly recall necessary background on separable functor and  comonad. 
Comonad as well as its dual concept  monad provides an important approach to  the category of equivariant objects with repect to an action of a finite group on an additive category. 
For example,  monad is used in \cite{chen, CCZ}, while  comonad is used  in \cite{E, E3} as well as here.
 In \S\ref{sec: equiv}, we collect mostly standard materials on equivariantization of an additive category $\A$ with respect to a $G$-action  on $\A$. 
In \S\ref{sec: functor}, we consider $G$-equivariant functors between additive categories with $G$-actions, i.e. those functors compatible with $G$-actions. 
We show that an adjoint of  a $G$-equivariant functor is naturally $G$-equivariant and   a $G$-equivariant functor that is fully faithful (resp. an equivalence) induces a fully faithful functor (resp. an equivalence) between the categories of equivariant objects. 
In \S\ref{sec: trivial}, we consider a trivial action and two related facts are included. 
In \S\ref{sec: reconstruction}, we show that if $k$ is algebraically closed, $G$ is solvable and $\A$ is idempotent-complete then we can obtain $\A$ from $\A^G$ after a finite sequence of equivariantization. This is a common generalization of Elagin's result \cite[Theorem 1.4]{E} in the special case when $G$ is abelian and Reiten and Riedtman's result \cite[Propsition 5.3]{RR} in the special case when $\A$ is the module category of a finite dimensional algebra. 

In \S\ref{sec: tri}, we  consider the possibility of a triangulated structure as well as a pretriangulated structure on $\A^G$ when $\A$ is triangulated. To achieve this, we introduce an admissible action of $G$ on $\A$, which we will assume, and define a canonical triangulated structure as well as a canonical pretriangulated structure on $\A^G$. We show that there is a unique canonical pretriangulated structure on $\A^G$. This was first observed by Chen \cite{XWC} in a slightly less general setting. Consequently, a canonical triangulated structure on $\A^G$ is unique once it exists.
To give instances of existence of canonical triangulated structures, we consider how equivariantization  commutes with localization and with the formation of the stable category of a Frobenius category. 
We show that if $\C$ is a triangulated subcategory of $\A$ closed under the action of $G$ then $(\A/\C)^G$ is canonically triangulated and there is an exact  equivalence up to retracts $\A^G/\C^G\ra (\A/\C)^G$ which is an equivalence when $\A^G/\C^G$ is idempotent-complete. 
For an idempotent-complete Frobenius category $\E$ with a $G$-action that preserves the exact structure on $\E$, we show that its stable category $\underline{\E}$ has an admissible action of $G$, $\underline{\E}^G$ is canonically triangulated and there is an exact equivalence up to retracts $\underline{\E^G}\ra \underline{\E}^G$ which is an equivalence when the stable category $\underline{\E^G}$ is idempotent-complete, where $\E^G$ is shown to be a Frobenius category in a canonical way. 
 We then apply these results to homotopy categories, derived categories and singularity categories to give some concrete examples in \S\ref{sec: exm}.
 We assume that $\A^G$ is canonically triangulated in \S\ref{sec: t-str}-\ref{sec: dim}. In \S\ref{sec: t-str} and \S\ref{sec: tilting}, we show how one can construct a t-structure on $\A^G$ (resp. a (weak) semi-orthogonal decomposition of $\A^G$,  an exceptional collection in $\A^G$,  a tilting object in $\A^G$) given one on $\A$ (resp. of $\A$,  in $\A$, in $\A$). 
 In \S\ref{sec: dim}, we show that $\A^G$ has the same dimension with $\A$ and that when $\A$ is  idempotent-complete and  saturated, $\A^G$ is also saturated.


\bigskip
\noindent  {\bf Notation and conventions.} Any ring in this article is assumed to be unital and associative. Modules over a ring always means a right module. We denote by $\Mod A$ (resp. $\mod A$, resp. $\proj A$) the category of arbitrary (resp. finitely generated, resp. finitely generated projective) $A$-modules. 

For a finite group $G$, $|G|$ denotes its order. 
Given a base field $k$,  a $G$-representation over $k$ refers to a group homomorphism $G\ra GL_k(V)^{\op}$, where $V$ is a $k$-space. If $|G|$ is invertible in $k$ then  $\irr(G)$ denotes a complete set of finite dimensional irreducbile $G$-representations over $k$. We denote by $G^*$ the group of characters over $k$.

When we write an adjoint pair $F: \A\rightleftarrows \B: H$ between two categories, we always mean that $F\colon \A\ra \B$ is  left adjoint to $H\colon \B\ra \A$.  An object $X$ in a category $\A$ is called a \emph{retract} of another object $Y\in \A$ if there is a morphsim $u\colon X\ra Y$ with a retraction (i.e., a morphism $v\colon Y\ra X$ such that $vu=\id_X$). 
A functor $F\colon \A\ra \B$ is said to be \emph{dense up to retracts} (resp. \emph{dense up to direct summands})
if each object in $\B$ is a retract  (resp. direct summand) 
of an object in the essential image of $F$. An equivalence up to retracts (resp. direct summands) is a functor that is fully faithful and dense up to retracts (resp. direct summands). 

  An integer $n$ is said to be invertible in an additive category $\A$ if for each morphism $f$ in $\A$, there exists a unique morphism $g$ such that $ng=f$. The unique morphism is also denoted by $\frac{1}{n}f$.  

\bigskip
\noindent {\bf Acknowledgements.} I thank Prof. Xiao-Wu Chen, my domestic supervisor,  for introducing to me the equivariant approach to weighted projective lines and for his reports on the same topic as in this article.  
I thank Jie Li  for discussion and  thank  Da-Wei Shen for constant and available-anytime communication. Moreover, I would like to take this opportunity to thank Prof. Lei Fu for answering my questions on algebraic geometry via email several years ago when I tried to understand the equivariant point of view to weighted projective lines yet knew little about algebraic geometry, and I thank also Prof. Bin Xu for helpful conversation on ramified Galois covers of the complex projective line. 

 I am currently visiting University of California, Los Angeles (UCLA) as a visiting student and I am supported by China Scholarship Council (CSC). I thank Prof. Ke-Feng Liu for accepting me  and for his hospitality. I thank also Prof. Ru-Gang Ye for helping me when I tried to apply the program of CSC.

\section{Equivariantization of additive categories}

\subsection{A reminder on separable functor and comonad}\label{sec: comonad}

In this subsection, we   review some necessary background on separable functors and comonads and our exposition mainly follows \cite[\S2-3]{XWC}. 

Following~\cite{NVVO}, a functor $F\colon \C\ra \B$ between two categories is called \emph{separable} if 
the natural transformation 
\[F\colon \hom_\C(-,-)\lra \hom_\B(F-,F-)\] 
between bifunctors from $\C^{\op}\times \C$ to the category of sets admits a retraction 
\[\Phi\colon \hom_B(F-,F-)\lra \hom_\C(-,-),\] 
i.e., a natural transformation $\Phi$  such that $\Phi F=\id$. 
The following fact characterizes when a left adjoint is separable (see \cite[Theorem 1.2]{Rafael},~\cite[Lemma 2.2]{XWC}). 
\begin{lem}\label{adjoint separable}
	Given an adjoint pair $F: \C\rightleftarrows \B: H$ with unit $\eta\colon \id \ra HF$ and counit $\epsilon\colon FH\ra \id$,  the following are equivalent:
\begin{enum2}
\item $F$ is separable;
\item there exists a natural transformation $\xi\colon HF\ra \id$ such that $\xi\eta=\id$;
\item there exist two natural transformations $\phi\colon HF\ra \id$ and  $\varphi\colon \id\ra HF$  such that $\phi\varphi=\id$. 
\end{enum2}
\end{lem}

Now we turn to comonad. We refer the reader to \cite{Maclane, BaWe}, where the dual concept monad is considered, for a detailed account.  Recall that a \emph{comonad} on a category $\B$ is a triple $T=(T,\epsilon, \Delta)$, where $T$ is an endofunctor of $\B$, $\epsilon\colon T\ra \id$ is a natural transform called the \emph{counit} and $\Delta\colon T\ra  T^2$ a natural transform called the \emph{comultiplication}, 
making the following diagrams commute 
\begin{equation}
	\begin{split}
		\xymatrix{T\ar[r]^\Delta\ar[d]_{\Delta} & T^2\ar[d]^{T\Delta} & & T \ar@{=}[dr]  & T^2 \ar[l]_{T\epsilon} \ar[r]^{\epsilon T} & T\ar@{=}[dl]\\ 
T^2\ar[r]^{\Delta T} & T^3 & & & T.\ar[u]^\Delta & }
\end{split}
\end{equation}
An adjoint pair $F: \C\rightleftarrows \B: H$ with unit $\eta\colon \id \ra HF$ and counit $\epsilon\colon FH\ra \id$  gives rise to a comonad $(FH, \epsilon, F\eta H)$ on $\B$. 

\begin{rmk}
	Dually, one can define a monad $(M, \eta, \mu)$ on a category $\A$ consisting of an endofunctor $M\colon \A\ra \A$, the \emph{unit} $\eta\colon \id\ra M$ and the \emph{multiplication} $\mu\colon M^2\ra M$, subject to the relations $\mu\circ M\mu=\mu\circ \mu M$ and $\mu\circ M\eta=\id_M=\mu\circ \eta M$. 
\end{rmk}

A comonad $(T, \epsilon, \Delta)$ is called \emph{separable} (see ~\cite[6.3]{BrVi} or \cite[2.9]{BBW})  if there exists a natural transformation $\sigma\colon T^2\ra T$ making the following diagrams 
\begin{equation}
	\begin{split}
		\xymatrix@C-15pt@R-15pt{T^2\ar[rr]^{T\Delta} \ar[dr]^\sigma \ar[dd]_{\Delta T} && T^3\ar[dd]^{\sigma T} & & T\ar@{=}[rr]\ar[ddr]_\Delta && T\\
& T \ar[dr]^\Delta & &&&&\\
T^3\ar[rr]^{T\sigma} && T^2, & & & T^2. \ar[uur]_\sigma & }
\end{split}
\end{equation}
commute. 

We already know that an adjoint pair yields a comonad and one solution (due to Eilenberg-Moore) of the converse problem is  the construction  of the category $\B_T$  of $T$-comodules for a given comonad $T=(T, \epsilon, \Delta)$ on $\B$. Recall that a $T$-comodule  is a pair $(X, f)$ consisting of an object $X\in \B$ and a morphism $f\colon X\ra TX$ in $\B$ making the following diagrams commute 
\begin{equation}
	\begin{split}
		\xymatrix{X\ar[r]^f\ar[d]_f & TX\ar[d]^{\Delta_X} & & X \ar[d]_f\ar@{=}[dr] &\\
TX\ar[r]^{Tf} & T^2 X & & TX\ar[r]^{\epsilon_X} & X. }
\end{split}
\end{equation}
A morphism $a\colon (X,f)\ra (X',f')$ of $T$-comodules is a morphism $a\colon X\ra X'$ rendering commutative the diagram 
\begin{equation}
	\begin{split}
		\xymatrix{X\ar[r]^f \ar[d]_a & TX\ar[d]^{Ta}\\  X'\ar[r]^{f'} & TX'.}
	\end{split}
\end{equation}
There is an adjoint pair $F_T: \B_T\rightleftarrows \B: H_T$. $F_T\colon \B_T\ra \B$ (called \emph{the forgetful functor}) is defined by the assignment $F_T((X,f))=X$ for an object $(X,f)\in \B_T$ and $F_T(h)=h$ for a morphism $h$ in $\B_T$; $H_T$ by the assignment $H_T(X)=(TX, \Delta_X)$ for an object $X\in \B$ and $H_T(h)=T(h)$ for a morphism $h$ in $\B$. The counit $\epsilon_T$ of the adjoint pair $(F_T,H_T)$ is $\epsilon_T=\epsilon$ and the unit $\eta_T$ is given by $(\eta_T)_{(X,f)}=f$ for $(X,f)\in \B_T$. Note that the comonad defined by the adjoint pair $(F_T, H_T, \eta_T, \epsilon_T)$ coincides with $T$. The forgetful functor $F_T$ is a separable functor iff $T$ is a separable comonad (see ~\cite[2.9(2)]{BBW}).

Let $F: \C\rightleftarrows \B :H$ be an adjoint pair with unit $\eta$ and counit $\epsilon$, and $T=(FH, \epsilon, F\eta H)$ the comonad that it defines in $\B$. Then there is a unqiue functor $K\colon \C\ra \B_T$ (see \cite[Theorem VI.3.1]{Maclane}), called \emph{the comparison functor}, such that $KH=H_T, F_TK=F$. In fact,  $K$  is defined such that  $K(X)=(FX, F\eta_X)$ for an object $X\in \C$ and $K(f)=F(f)$ for a morphism $f$ in $\C$.  $K$ is fully faithful on the essential image of $H$ (see e.g. \cite[Lemma 3.3]{XWC}). The functor $F$ is called \emph{comonadic} (or \emph{of effective descent type}) resp. \emph{precomonadic} (or \emph{of descent type}) if $K$ is an equivalence resp. fully faithful. A characterization when $F$ is precomonadic or comonadic is given by  Beck's theorem. We don't try to recall it here since we will not need it and we refer the reader to ~\cite{Beck}, ~\cite[Thereom 2.1]{Mesa}, ~\cite[Theorem 3.3.4]{BaWe} and ~\cite[Theorem VI.7.1]{Maclane}, where the conclusion for the dual concept monad is stated. 

The following fact (see \cite[Proposition 3.5]{XWC})  relates separability of $F$ to separability of $T$ and (``almost'') comonadicity of $F$. 
\begin{prop}\label{separable comonadicity}
Let $F: \C\rightleftarrows \B :H$ be an adjoint pair with unit $\eta$ and counit $\epsilon$, and $T=(FH, \epsilon, F\eta H)$ the comonad that it defines on $\B$. 
 Then $F$ is separable iff $T$ is separable and if the comparison functor $K\colon \C\ra  \B_T$ is an equivalence up to retracts. 
\end{prop}
Here is an immediate corollary (see \cite[Corollary 3.17, Proposition 3.18]{Mesa}, \cite[Corollaries 3.10-3.11]{E4}, \cite[Corollary 3.6]{XWC}). 
\begin{cor}
	\label{comonadicity2}
	With the same notation as in Proposition~\ref{separable comonadicity}, assume further $\C$ is idempotent-complete. Then $F$ is separable iff $T$ is separable and if $F$ is comonadic. 
\end{cor}

\subsection{Recollections on equivariantization}\label{sec: equiv}

Let $k$ be a field, $\A$ an additive category and $G$ a finite group with unit $e$. 

\begin{defn}
	A (\emph{right}) \emph{action} of $G$ on an additive category $\A$ means the following data:

	\begin{itemize}
		\item a collection $\{F_g\mid g\in G\}$ of  autoequivalences of $\A$;  
		\item a collection $\{\delta_{g,h}\colon F_gF_{h}\ra F_{hg}\mid g,h\in G\}$ of natural isomorphisms  satisfying a $2$-cocyle condition, i.e., such that for each $g,h,l\in G$, the following diagram commutes
	\begin{equation}
		\begin{split}
			\xymatrix{F_gF_hF_l\ar[d]_{\delta_{g,h}F_l} \ar[r]^{F_g\delta_{h,l}} & F_gF_{lh}\ar[d]^{\delta{g,lh}}\\
			F_{hg}F_l\ar[r]^{\delta_{hg,l}} & F_{lhg}.}
		\end{split}
	\end{equation}
	\end{itemize}

\end{defn}
We will denote such a $G$-action by $\{F_g, \delta_{g,h}\mid g,h\in G\}$. If each $F_g$ is an automorphism and $\delta_{g,h}=\id$ for each $g,h\in G$ then we say that the action is a \emph{strict} action. Note that a strict action of $G$ on $\A$ amounts to a group homomorphism $G\ra \Aut(\A)$, where $\aut(\A)$ is the group of automorphisms of $\A$. 
When $\A$ is $k$-linear, we will insist that the action be $k$-linear, i.e., $F_g$'s are $k$-linear. So we will not mention a $k$-linear action explicitly. 

\begin{rmk} Here we record a point of view  in the language of monoidal categories (see \cite{DGNO}). Let $\underline{G}$ be the monoidal category: the objects of $\underline{G}$ are elements of $G$, the only morphisms are the identities and the tensor product is given by multiplication in $G$. 
	Let $\underline{\End}(\A)$ denote the category of ($k$-linear) additive endofunctors of  ($k$-linear)  $\A$, which is a monoidal ($k$-linear) category with composition of functors as its tensor product. A ($k$-linear) action of $G$ on $\A$ is the same as an action of $\underline{G}$ on $\underline{\End}(\A)$, that is, we are given a  monoidal functor $F\colon \underline{G}\ra \underline{\End}(\A)$.  
\end{rmk}
\begin{rmk}\label{left-right}
One can similarly define a left action of $G$ on $\A$, that is, a collection of $\{E_g\mid g\in G\}$ of autoequivalences of $\A$ and a collection $\{\gamma_{g,h}\colon E_gE_h\ra  E_{gh} \mid g,h\in G\}$ of natural isomorphisms satisfying a similar $2$-cocyle condition. The assignment \[E_g=F_{g^{-1}},\quad \gamma_{g,h}=\delta_{g^{-1},h^{-1}},\quad g,h\in G\] allows one to transfer a right action $\{F_g,\delta_{g,h}\}$ to a left action  $\{E_g, \gamma_{g,h}\}$ and vice versa. 
\end{rmk}

Given a $G$-action $\{F_g, \delta_{g,h}\mid g,h\in G\}$ on $\A$, there exists a unique natural isomorphism $\Xi\colon F_e\overset{\cong}{\ra} \id$, which is called the \emph{unit} of the action, such that $F_e\Xi=\delta_{e,e}=\Xi F_e$. For $g\in G$, we have $F_g \Xi=\delta_{g,e}, \Xi F_g=\delta_{e,g}$. 
Moreover, for each $g\in G$, we have an adjoint pair $(F_g, F_{g^{-1}})$ with unit   $\delta_{g^{-1}, g}^{-1}\circ \Xi^{-1}$ and counit $\Xi\circ \delta_{g,g^{-1}}$.

The following commutative diagrams are necessary in checking commutativity of diagrams occuring later.
\begin{lem}\label{associativity}
For $g,h,k,l\in G$, the following  diagram 	
\begin{equation}
	\begin{split}
		\xymatrix{F_gF_hF_kF_l\ar[r]^{F_gF_h\delta_{k,l}}\ar[d]_{\delta_{g,h}F_kF_l} & F_gF_hF_{lk}\ar[d]^{\delta_{g,h}F_{lk}}\\
		F_{hg}F_kF_l\ar[r]^{F_{hg}\delta_{k,l}} & F_{hg}F_{lk}}
	\end{split}
\end{equation}
commutes. 
\end{lem}
\begin{proof}
	This follows readily from the $2$-cocyle condition satisfied by $\{\delta_{g,h}\}$. 
\end{proof}

\begin{cor}\label{commutativity}
	For $g,h\in G$, we have a commutative diagram of natural isomorphisms
	\begin{equation}\label{eq: commutativity}
		\begin{split}
			\xymatrix@C-10pt{\hom(F_gF_h A, B)\ar[r]\ar[d] & \hom(F_h A, F_{g^{-1}} B)\ar[r] & \hom(A, F_{h^{-1}} F_{g^{-1}} B)\ar[d] \\
	\hom(F_{hg} A, B)\ar[rr] & & \hom(A, F_{(hg)^{-1}} B),}
\end{split}
\end{equation}
where each horizontal arrow is given by the relevant adjoint pair and the left (resp. right) vertical arrow is induced by $\delta_{g,h}^{-1}$ (resp. $\delta_{h^{-1}, g^{-1}}$).
\end{cor}
\begin{proof}
One can use Lemma~\ref{associativity}. 
	We leave the details to the reader.
\end{proof}

Suppose  $\A$ admits a $G$-action $\{F_g, \delta_{g,h}\}$ with unit $\Xi$. A \emph{$G$-equivariant} object of $\A$ is a pair $(A,\{\lambda_g^A\mid g\in G\}),$ where  $A$ is an object in $\A$  and $\{\lambda^A_g\mid g\in G\}$ is  a collection $\{\lambda_g^A\colon F_g(A)\ra A\}$ of isomorphisms, called an \emph{equivariant structure}  on $A$, such that for each $g,h \in G$, the following diagram commutes
\begin{equation}
	\begin{split}
		\xymatrix{ F_gF_h A \ar[d]_{\delta_{g,h,A}}\ar[r]^{F_g(\lambda_h^A)} & F_gA\ar[d]^{\lambda_g^A}\\
		F_{hg}A  \ar[r]^{\lambda_{hg}^A} & A.}
\end{split}
\end{equation}
	In this case, we have  $\lambda_e^A=\Xi_A$. 
By abuse of language, we will also say that $A\in \A$ is $G$-equivariant with its equivariant structure implicit. An object $B\in \A$ is called \emph{$G$-invariant} if $F_g(B)\cong B$ for all $g\in G$. If $B\in \A$ is $G$-equivariant then obviously it is $G$-invariant.

Let $(A, \{\lambda_g^A\}), (B, \{\lambda_g^B\})$ be $G$-equivariant objects of $\A$. Then there is a right action of $G$ on $\hom_\A(A,B)$ defined by \[f.g=\lambda_g^B\circ F_g(f)\circ (\lambda_g^A)^{-1}\] for $g\in G, f\in \hom_\A(A,B)$. 
Moreover, $G$ acts on the endomorphism ring $\End_\A(A)$ of $A$ by automorphism of a ring. 
The category $\A^G$ of equivariant objects consists of $G$-equivariant objects  of $\A$ and 
a morphism  $f\colon (A, \{\lambda_g^A\}) \ra (B, \{\lambda_g^B\})$ in $\A^G$ is a morphism $f\in \hom_{\A}(A,B)$ that is invariant under the induced action of $G$ on $\hom_\A(A,B)$. $\A^G$ is called the \emph{equivariantization} of $\A$ with respect to the action of $G$. $\A^G$ is also an additive category. If $\A$ is an abelian category then so is $\A^G$.

Suppose $\A$ is $k$-linear. Given a $G$-representation $\rho\colon G\ra GL_k(V)^{\op}$, we have a $k$-linear endofunctor \[-\otimes \rho\colon \A^G\lra \A^G,\quad (A,\{\lambda_g^A\})\mapsto (A\otimes V, \{\lambda_g^A\otimes \rho_g\})\] sending a morphism $f$ in $\A^G$ to $f\otimes \id_V$.  ($\otimes$ is always short for $\otimes_k$ in this article.) 

We need the following standard fact. 
\begin{lem}\label{tensor hom} If $\A$ is $k$-linear  and  $A,B\in \A$ are $G$-equivariant then for finite dimensional $G$-representations $V,W$ over $k$, we have an isomorphism of $G$-representations
	\[\hom_\A(A\otimes V, B\otimes W)\cong \hom_\A(A,B)\otimes \hom_k(V,W).\]
\end{lem}
\begin{proof}
	This follows from  direct verification.
\end{proof}
We have the following
\begin{cor}[{cf. \cite[Lemma 4.1]{BKR}}]
	If $\A$ is $k$-linear, $|G|$ is invertible in $k$ and  $A,B\in \A$ are $G$-equivariant then we have a decomposition of $\hom_\A(A,B)$ as a $G$-representation \[\hom_\A(A,B)\cong \bigoplus_{\rho\in\irr(G)} \hom_{\A^G}(A\otimes \rho, B)\otimes \rho.\]
\end{cor}
\begin{proof}

	For a finite dimensional $G$-representation $V$, we have isomorphisms of $G$-representations \[\hom_k(V, \hom_\A(A,B))\cong V^\vee\otimes \hom_\A(A,B)\cong \hom_\A(A\otimes V, B).\] The desired follows readily.
\end{proof}

Here are two standard examples of equivariantization.

\begin{exm}\label{exm skew}
	Let $R$ be a  ring and $G$ a finite group. Suppose $R$ admits a right $G$-action by automorphisms, that is, we are given a group homomorphism $\rho\colon G\ra \Aut(R)^{\op}$. We also write $\rho(g)(a)=a^g$ for $g\in G, a\in R$. 
	For $g\in G$ and $M\in \Mod R$, the \emph{twisted module} $M^g$ is defined such that $M^g=M$ as an abelian group and the new $R$-action $._g$ is defined by $m._ga=m.a^{g^{-1}}$ for $m\in M$ and $a\in R$. This yields an automorphism $(-)^g\colon \Mod R\ra \Mod R$ sending a morphism $f$ in $\Mod R$ to $f$, which is actually isomorphic to extension by scalars along the ring isomorphism $\rho(g)\colon R\ra R.$ Then we have a strict $G$-action $\{F_g=(-)^g\}$ on $\Mod R$. 
	Recall also that we have a skew group algebra $RG$, which is a free left $R$-module with elements in $G$ as a basis and has multiplication defined by $(ag)(a'g'):=aa'^{g^{-1}}gg'$ for $a,a'\in R, g,g'\in G$. There is a canonical isomorphism \[\Phi\colon \Mod RG\overset{\cong}{\lra} (\Mod R)^G\] such that  $\Phi(M)=(M, \{\lambda_g^M\})$ for an object $M\in \Mod RG$, where $\lambda_g^M$ is defined by \[\lambda_g^M\colon M^g\ra M,\quad  m\mapsto m.g,\] and $\Phi(f)=f$ for a morphism $f$ in $\Mod RG$.   $\Phi$ restricts to an isomorphism $\mod RG\cong (\mod R)^G$; if $|G|$ is invertible in $R$, $\Phi$  restricts to  an equivalence $ \proj RG\simeq (\proj R)^G$.
\end{exm}
\begin{exm}
	Let $X$ be a  scheme of finite type over  a field $k$ 
	and let  $\coh(X)$ be the category of coherent sheaves on $X$. Suppose $X$ admits a right $G$-action by $k$-automorphisms, that is,  we are given a group homomorphism $\rho\colon G\ra \aut_k(X)^{\op}$, where $\aut_k(X)$ is the group of  automorphisms of $X$ as a $k$-scheme. Then there is a strict $G$-action $\{F_g\}$  on $\coh(X)$, where $F_g=\rho(g)_*$ is pushout along the $k$-automorphism $\rho(g)\colon X\ra X$. 
\end{exm}

There are an obvious forgetful functor $\omega\colon \A^G\ra \A$, forgetting the $G$-equivariant structure, and  an induction functor $\ind\colon \A\ra \A^G$ such that  $\ind(A)$ is $\oplus_{g\in G} F_g(A)$ equipped with an obvious equivariant structure. Sometimes we will denote the induction functor by $\ind^G$ to avoid possible confusion. 

The following properties  are essential for us.
\begin{lem}
	\label{adjoint}
	\begin{enum2}
	\item $\omega$ is a faithful conservative additive functor. $\ind$ is left and right adjoint to $\omega$ such that the composition $X\ra \ind\circ \omega(X) \ra X$  of adjunctions  equals $|G|\cdot \id_X$. 
	
	\item $\ind$ is a separable functor. If $|G|$ is invertible in $\A$   then $\omega$ is a separable functor.
	\item If $\A$ is idempotent-complete and $|G|$ is invertible in $\A$ then each object $X\in \A^G$ is a direct summand of $\ind\circ \omega(X)$ in $\A^G$.  
	\item If $\A$ is idempotent-complete then so is $\A^G$. 
	\end{enum2}
\end{lem}
\begin{proof}
	For (1), see e.g. \cite[Lemma 4.6]{DGNO}. (3) follows from (1). That $\omega$ is separable under the assumption that $|G|$ is invertible in $\A$ follows from (1) and Lemma~\ref{adjoint separable}. The argument for $\ind$ being separable is similar since the composition $A\ra \omega\circ \ind(A)\ra A$ of adjunctions is $\id_A$. (4) is \cite[Lemma 2.3]{chen}.
\end{proof}

	Let $\C$ be a strictly full additive subcategory of $\A$. We say that $\C$ is $G$-invariant if for $g\in G$, $F_g(\C)\subset \C$, equivalently, for $g\in G$, $F_g(\C)=\C$.
	In this case, the $G$-action on $\A$ restricts to an action on $\C$. $\C^G$ is naturally identified with the  full subcategory $\omega^{-1}(\C)$ of $\A^G$. Thus we will not distinguish $\C^G$ with $\omega^{-1}(\C)$. If $A\in\A$ is $G$-invariant then we have $\add \ind(A)=(\add A)^G,$ where $\add A$ refers to the full subcategory of $\A$ consisting of direct summands of finite direct sums of $A$.

We will also need to consider restriction of a $G$-action on $\A$ to an action of a subgroup $H$ of $G$.  Let $G/H=\{Hg_1,\dots, Hg_n\}$, where $\{g_1=e,\dots, g_n\}$ is a fixed complete set of representatives of right cosets of $H$ in $G$.  
The $G$-action  on $\A$ restricts to an $H$-action  on $\A$ in an obvious way. We have an obvious restriction functor $\res^H_G\colon \A^G\ra \A^H$, generalizing the forgetful functor above, and an induction functor $\ind_H^G\colon \A^H\ra \A^G$,  generalizing the induction functor above, such that $\ind((A, \{\lambda_h^A\mid h\in H\}))$ is $\oplus_{i=1}^n F_{g_i}(A)$ endowed with an obvious equivariant structure. 
Still, we have the following similar fact.
\begin{lem}
	\begin{enum2}
	\item The triple  $(\ind_H^G, \res^H_G, \ind_H^G)$ is an adjoint triple  such that   the composition $U \ra \res^H_G\circ \ind_H^G(U)\ra U$ of adjunctions  is $\id_U$ and the composition $X\ra \ind_H^G\circ \res^H_G(X)\ra X$ of adjunctions  is $|G/H|\cdot \id_X$.
	\item $\ind^G_H$ is a separable functor; if $|G/H|$ is invertible in $\A$ then $\res^H_G$ is a separable functor.
	\end{enum2}
\end{lem}

We now investigate the relation between equivariantization and idempotent-completion. 
Recall that  the idempotent-completion $\widehat{\A}$ of $\A$ is defined as follows:    
objects of $\widehat{\A}$ are pairs $(A,a)$, where $A\in \A$ and $a\in \hom_\A(A,A)$ is an idempotent; a morphism $f\colon (A_1,a_1)\ra  (A_2, a_2)$ in $\widehat{\A}$ is a morphism $f\colon A_1\ra A_2$ in $\A$ satisfying $f=a_2fa_1.$ We have a canonical functor $\iota\colon \A\ra \widehat{\A}$ sending an object $A\in \A$ to $(A,1)\in \widehat{\A}$ and sending a morphism $f$ in $\A$ to $f$, which is an equivalence up to direct summands. Given an additive functor $F\colon  \A\ra \B$ between two additive categories, there is an  induced additive functor $\widehat{F}\colon \widehat{\A}\ra \widehat{\B}$ sending an object $(A,a)\in \widehat{\A}$ to $(F(A), F(a))\in \widehat{\B}$ and sending a morphism $f$ in $\widehat{\A}$ to $F(f)$. $\widehat{F}$ is an equivalence iff $F$ is an equivalence up to retracts.  

Observe that a $G$-action $\{F_g, \delta_{g,h}\}$ on $\A$ induces a $G$-action $\{\widehat{F}_g, \widehat{\delta}_{g,h}\}$ on  $\widehat{\A}$, where  $\widehat{F}_g\colon \widehat{\A}\ra \widehat{\A}$ is induced by the autoequivalence $F_g\colon \A\ra \A$ and $\widehat{\delta}_{g,h}$ is defined by $\widehat{\delta}_{g,h,(A,a)}=\delta_{g,h,A}\circ F_gF_h(a)$.

	\begin{lem}
		\label{idem}
There is a fully faithful functor  $  \widehat{\A^G}\ra \widehat{\A}^G $ that is an equivalence when $|G|$ is invertible in $\A$. 
\end{lem}

\begin{proof}
	Define a functor $F\colon \widehat{\A^G} \ra   \widehat{\A}^G$ such that \[F(((A, \{\lambda_g^A\}),a))=((A,a), \{\lambda_g^{(A,a)}=a\circ \lambda_g^A\})\] and such that $F(f)=f$ for a morhpim $f$ in $\widehat{\A}^G$. $F$ is well-defined and one readily sees that $F$ is fully faithful. Now suppose $|G|$ is invertible in $\A$.  By Lemma~\ref{adjoint}(4), $\widehat{\A}^G$ is idempotent-complete. Then each object $((A,a), \{\lambda_g^{(A,a)}\})$ in $\widehat{\A}^G$ is a direct summand of $\ind((A,a))$. 
Since $\widehat{\A^G}$ is idempotent-complete, to see that $F$ is an equivalence, it sufficies to note that \[\ind((A,a))=F((\ind(A), \ind(a)))\] lies in the essential image of $F$. 
\end{proof}

\begin{rmk}
	Given an action of $G$ on $\A$, one can associate  another category $\A/G$, the orbit category. $\A/G$ has the same objects with $\A$. For $A, B\in \A/G$, \[\hom_{\A/G}(A,B):=\oplus_{g\in G}\hom_\A(A, F_g(B)).\] Composition of morphisms is defined in an obvious way.  It's shown in \cite[Proposition 2.4]{chen} that there is a  fully faithful comparison functor \[\ind\colon \A/G\lra \A^G, \quad X\mapsto \ind(X),\] which is an equivalence up to retracts if  $|G|$  is invertible in $\A$. 
	In constrast, a severe problem with $\A/G$ is that $\A/G$ need not be idempotent-complete even if $\A$ is. 
\end{rmk}

\subsection{G-equivariant functors}\label{sec: functor}
 For additive categories with $G$-actions, it's very natural to consider functors between them that are compatible with $G$-actions.  In the language of monoidal categories, such a functor bears the name of a module functor~\cite{Ostrik}. 

Let $\A$, $\B$  and $\C$ be additive categories equipped with $G$-actions $\{M_g, \alpha_{g,h}\}$, $\{N_g, \beta_{g,h}\}$ and $\{P_g, \gamma_{g,h}\}$, respectively.
\begin{defn} 
	\begin{enum2}
\item  An additive functor $F\colon \A\ra \B$ is called $G$-equivariant if there exists a collection $\{\phi_g\colon  FM_g  \overset{\cong}{\ra}  N_gF\mid g\in G\}$ of natural isomorphisms such that for all $g,h\in G$ the following diagrams commute 
	\begin{equation}\label{equiv functor}
		\begin{split}
			\xymatrix@C-10pt{FM_gM_h  \ar[r]^{\phi_gM_h}\ar[d]_{F\alpha_{g,h}} & N_gFM_h \ar[r]^{N_g\phi_h} &    N_gN_hF\ar[d]^{\beta_{g,h}F} & &  FM_e \ar[rr]^{\phi_e} \ar[dr]_{F\Xi} && N_e F \ar[dl]^{\Upsilon F}\\
	FM_{hg}  \ar[rr]^{\phi_{hg}} & &  N_{hg}F, && & F, &}
\end{split}
\end{equation}
where $\Xi$ resp. $\Upsilon$ is the unit of the $G$-action on $\A$ resp. $\B$. 

\item Let $F_1, F_2\colon \A\ra \B$ be two  $G$-equivariant additive functors with equivariant structures $\{\phi^1_g\}$, $\{\phi^2_g\}$, respectively. Let $\theta\colon F_1\ra F_2$ be a natural transform. $\theta$ is called $G$-equivariant if for $g\in G$, the following diagram commutes 
	\begin{equation}
		\begin{split}
			\xymatrix{F_1M_g\ar[r]^{\theta M_g} \ar[d]_{\phi^1_g} & F_2M_g \ar[d]^{\phi_g^2}\\
	N_gF_1\ar[r]^{N_g\theta} & N_gF_2.}
\end{split}
\end{equation}
\end{enum2}
\end{defn}

Some immediate observations are in order:
\begin{enum2}
\item A $G$-equivariant functor $F\colon \A\ra \B$ induces an additive functor $F^G\colon \A^G\ra \B^G$. For $(A,\{\lambda^A_g\})\in \A^G$, $F^G((A, \lambda_g^A))=(FA, \{\kappa_g^{FA}\})$, where $\kappa_g^{FA}=  F(\lambda^A_g)\circ \phi_{g,A}^{-1}$; for a morphism $f$ in $\A^G$, $F^G(f)=F(f)$. We have an equality $\omega\circ F^G=F\circ \omega$ and a natural isomorphism $\oplus\phi_g\colon    F^G\circ \ind \overset{\cong}{\ra}    \ind \circ F $. 

\item If $\A_1$ is a $G$-invariant additive subcategory of $\A$ and  $\iota\colon \A_1\monic \A$ is the inclusion functor then $\iota$ is naturally $G$-equivariant. In particular, $\id_\A$ is naturally $G$-equivariant and we have $\id_\A^G=\id_{\A^G}$.

\item Let $F\colon \A\ra \B, H\colon \B\ra \C$ be two $G$-equivariant functors with equivariant structures $\{\phi_g\colon F M_g   \overset{\cong}{\ra}  N_g F \}$, $\{\varphi_g\colon  HN_g  \overset{\cong}{\ra} P_g H  \}$, respectively. For $g\in G$, define $\tau_g= \varphi_g F   \circ H\phi_g $. It's direct to check that $\{\tau_g\}$ makes $H\circ F$ a $G$-equivariant functor and that $(H\circ F)^G=H^G\circ F^G$.  

\item Let $F_1, F_2\colon \A\ra \B$ be two $G$-equivariant functors and $\theta\colon F_1\ra F_2$ a $G$-equivariant natural transformation. $\theta$ induces a natural transformation $\theta^G\colon F_1^G\ra F_2^G$ such that $\theta^G_{(A,\{\lambda_g^A\})}=\theta_A$. 
\end{enum2}

Now we show that an adjoint of a $G$-equivariant funcor is also $G$-equivariant. This may remind one of the fact that an adjoint of an exact functor between triangulated categories is also exact.

\begin{lem}\label{equiv adjoint}
	Let $\A,\B$ be additive categories with $G$-actions. Suppose $F: \A\rightleftarrows \B: H$ is an adjoint pair. If $F$ is $G$-equivariant then  $H$ is naturally $G$-equivariant for which the unit $\nu\colon \id \ra HF$ and the counit $\mu\colon FH\ra \id$ are $G$-equivariant. Moreover, $(F^G, H^G)$ is an adjoint pair with unit $\nu^G$ and counit $\mu^G$. A similar result holds if $H$ is $G$-equivariant.
\end{lem}
\begin{proof}
	In this proof, we will sometimes write a morphism set $\hom(C,D)$ as $(C, D)$ for short and one will see that there will be no ambiguity about which category the morphism set belongs to. 
	Let $\{M_g, \alpha_{g,h}\}$ resp. $\{N_g, \beta_{g,h}\}$ be the $G$-action on $\A$ resp. $\B$ with unit $\Xi$ resp. $\Upsilon$. For $g\in G$, let \[\alpha_g=\Xi\circ \alpha_{g^{-1},g},\quad \beta_g=\Upsilon\circ \beta_{g^{-1},g}.\]

	Suppose $\{\phi_g\colon FM_g\overset{\cong}{\ra} N_g F\}$ makes $F$ $G$-equivariant. 
For $g\in G$,  the composition of the following  isomorphisms 
	\[(A,HN_gB)  \ra (FA, N_gB)\ra (N_{g^{-1}}FA, B) 
	\ra (FM_{g^{-1}}A, B)\ra  (M_{g^{-1}}A, HB) \ra (A, M_gH B)\]
defines a natural isomorphism $\sigma_g\colon H N_g\ra M_g H$. 
We will show that $\{\sigma_g\}$ makes $H$ $G$-equivariant.  Consider the following diagram
	\[\begin{footnotesize}\xymatrix@R-10pt@C-20pt{(FM_gM_h A, B)\ar[rr]\ar[dr]\ar[dd] && (N_g F M_h A, B)\ar[rr]\ar@{-}[d] && (N_gN_h F A, B)\ar[dr]\ar@{-}[d] &\\
		& (F M_{hg} A, B)\ar[rrrr] \ar[dd] &\ar[d]&&\ar[d] & (N_{hg} F A, B)\ar[dd]\\
		(A, M_{h^{-1}} M_{g^{-1}} H B)\ar@{-}[r]\ar[dr] &\ar[r]&(A, M_{h^{-1}}HN_{g^{-1}} B)\ar[rr] && (A, H N_{h^{-1}}N_{g^{-1}} B)\ar[dr] &\\
		& (A, M_{(hg)^{-1}} H B)\ar[rrrr] &&&& (A, HN_{(hg)^{-1}} B),}\end{footnotesize}\]
where each vertical arrow is the  adjunction map of the corresponding adjoint pair,   the top rectangle is induced by the rectangular diagram of \eqref{equiv functor} and the bottom rectangle is induced by a similar diagram for $H$ involving certain $\sigma_g$'s. 
Since $F$ is $G$-equivariant, the top rectangle commutes.  
One can see that the two squares on the left and right side commute, where one needs to use Corollary~\ref{commutativity}, and that the  rectangles on the front and back side commute, where one needs to use the definition of $\sigma_g$'s.   So the bottom rectangle commutes. This establishs the rectangular commutative diagram for $H$ to be $G$-equivariant. The triangular commutative diagram for $H$ to be $G$-equivariant is easy to be established. Thus $H$ is indeed $G$-equivariant. 

Now we show that  the unit $\nu\colon \id \ra HF$ is $G$-equivariant. For $A\in \A$ and $g\in G$, since $F$ is $G$-equivariant, we have a commutative diagram 
\[\xymatrix@C+10pt{(FA, FA)\ar[r]^{-\circ F\alpha_{g,A}} \ar[d]_{N_g} & (FM_{g^{-1}}M_g A, FA)\ar[r]^{-\circ \phi^{-1}_{g^{-1}, M_g A}} & (N_{g^{-1}}FM_g A, FA)\ar[d]\\
(N_g FA, N_g FA)\ar[rr]^{-\circ \phi_{g, A}} && (FM_g A, N_g FA),}\]
where the unadorned arrow is given by the corresponding adjoint pair.  Moreover, there is a unique morphism \[\theta_{g,A}\colon \hom(FA, FA)\lra \hom(FM_g A, FM_g A)\] rendering commutative the  following diagram
\[\xymatrix{(FA, FA)\ar[r]^{N_g}\ar[d]_{\theta_{g,A}} & (N_g FA, N_g FA)\ar[d]^{-\circ \phi_{g,A}}\\
(FM_g A, FM_g A)\ar[r]^{\phi_{g, A}\circ -} & (FM_g A, N_g FA).}\]
Note that $\theta_{g,A}(\id_{FA})=\id_{FM_gA}$. 
Then it is ready to see the existence of  the following  commutative diagram 
\[\begin{footnotesize}\xymatrix@R-10pt{(M_gA, M_g HF A)\ar@{=}[r]  & (M_g A, M_g HF A) \ar[d]\ar@{=}[r] &  (M_g A, M_g HF A)\ar@/^3.5pc/[dddd]^{\sigma_{g,FA}\circ -}\\
	(A, HF A)  \ar[d]\ar[u]^{M_g}\ar[r]^{-\circ \alpha_{g, A}} & (M_{g^{-1}}M_g A, HF A) \ar[d] & \\
	(FA, FA)  \ar[r]^{-\circ F\alpha_{g, A}} \ar[d]_{\theta_{g, A}} & (FM_{g^{-1}}M_g A, FA)\ar[r]^{-\circ \phi^{-1}_{g^{-1}, M_gA}} & (N_{g^{-1}}FM_g A, FA)\ar[d] \\
	(FM_g A, FM_g A)\ar[rr]^{\phi_{g, A}\circ -} \ar[d] && (FM_g A, N_g FA)\ar[d] \\
	(M_g A, HFM_g A)\ar[rr]^{H\phi_{g,A}\circ -} && (M_g A, HN_g F A),}
\end{footnotesize}\]
where each unadorned arrow is given by the corresponding adjoint pair. 
It follows that we have $H\phi_g\circ \nu M_g=\sigma_g F\circ M_g \nu$ for $g\in G$. That is, $\nu$ is $G$-equivariant. So we have proved that $H$ is $G$-equivariant and the unit $\nu$ of the adjoint pair $(F, H)$  is $G$-equivariant if $F$ is $G$-equivariant. 
Then by duality, we can see that the counit $\mu$ is also $G$-equivariant.

And then it is easy to show that the adjunction map \[\hom_\B(FA, B) \overset{\cong}{\lra} \hom_\A(A, H B)\] is $G$-equivariant if $A$, $B$ are equivariant and $FA$, $HB$ are endowed with the induced equivariant structure.  It follows that we have an adjoint pair \[\xymatrix{F^G: \A^G\ar@<0.5ex>[r] & \ar@<0.5ex>[l] \B^G  :H^G}\]  with unit $\nu^G$ and counit $\mu^G$. 

By duality, we see that a similar result holds if $H$ is $G$-equivariant.
\end{proof}

A $G$-equivariant functor that is fully faithful (resp. an equivalence) induces a fully faithful functor (resp. an equivalence). 
\begin{prop}\label{G equiv}
	Suppose $\A$, $\B$ are additive categories with $G$-actions. 
	Let $F\colon \A\ra \B$ be a $G$-equivariant functor. 
	 If $F\colon \A\ra \B$ is an equivalence (resp. fully faithful)  then so is $F^G\colon \A^G\ra \B^G$. 

\end{prop}
\begin{proof}
	The assertion when $F$ is fully faithful is obvious. The assertion when $F$ is an equivalence follows readily from  this and Lemma~\ref{equiv adjoint}. 
\end{proof}

Here are  examples of $G$-equivariant functors.

\begin{exm}\label{exidem}
	Suppose $\A$ has a $G$-action. Let  the idempotent-completion $\widehat{\A}$ of $\A$ be endowed with the induced $G$-action. 
	It's easy to check that the canonical functor $\iota\colon \A\ra \widehat{\A}$ is $G$-equivariant. Thus we  have a fully faithful functor $\iota^G\colon \A^G\ra \widehat{\A}^G$. 
\end{exm}

\begin{exm} \label{equiv skew group}

	Suppose $\A$ is idempotent-complete and  $\{M_g, \alpha_{g,g'}\}$ is a $G$-action on $\A$ with unit $\Xi$. Let $(A,\{\lambda_g^A\})\in \A^G$. Recall that we have an equivalence \[\hom_\A(A,-)\colon \add A\overset{\simeq}{\lra} \proj \End_\A(A).\] 
	If we equip $\proj \End_\A(A)$ with the $G$-action $\{N_g=(-)^g\}$ induced by the $G$-action  on the ring $\End_\A(A)$ (see Example~\ref{exm skew}) then $\hom_\A(A,-)$ is $G$-equivariant. 
  Indeed, 
  for $B\in \add A$, define a mapping \[\phi_{g,B}\colon \hom_\A(A, M_g(B))\lra N_g(\hom_\A(A,B))\] by  \[\phi_{g,B}(h)=\Xi_{B}\circ \alpha_{g^{-1},g,B} \circ M_{g^{-1}}(h)\circ (\lambda_{g^{-1}}^A)^{-1}.\]
  One easily sees that $\phi_{g,B}$ is an isomorphism of right $\End_\A(A)$-modules and  is functorial in $B$, whereby yielding a natural isomorphism \[\phi_g\colon \hom_\A(A, M_g(-))\lra N_g\circ \hom_\A(A,- ),\] and $\{\phi_g\}$ makes $\hom_\A(A,-)$  $G$-equivariant. Then we have an equivalence \[\hom_\A(A,-)^G\colon (\add A)^G \overset{\simeq}{\lra} (\proj  \End_\A(A))^G.\] 
  If $|G|$ is invertible in $\A$ then we have an equivalence \[\add \ind(A)=(\add A)^G\simeq \proj \End_\A(A)G\]
  and thus $\End_{\A^G}(\ind(A))$ is Morita equivalent to the skew group ring $\End_\A(A)G$. 
  
  Suppose additionally that $\A$ is $k$-linear and $\End(A)=k$. By the discussion above, we have \[(\add A)^G=\add \{A\otimes \rho\mid \text{finite dimensional indecomposable $kG$-modules $\rho$}\},\] where the second $A$ is equipped with any fixed equivariant structure.    
  If $|G|$ is invertible in $k$ then we have \[(\add A)^G=\add \bigoplus_{\rho\in \irr(G)}A\otimes \rho=\add A\otimes kG,\] where $kG$ is the regular representation of $G$, and for  $\rho, \rho'\in \irr(G)$, by Lemma~\ref{tensor hom}, we have \[\hom_{\A^G}(A\otimes\rho, A\otimes \rho')=\left\{\begin{matrix} k & \text{if $\rho=\rho'$}\\ 0 & \text{otherwise.}\end{matrix}\right.\]

\end{exm}

\begin{exm}
	Let $\A$ be a Hom-finite additive $k$-category, where $k$ is a field. Recall from ~\cite{BK, ReiVan} that  a Serre functor  of $\A$ is an autoequivalence $\SS$  of $\A$ such that there is a bifunctorial isomorphism \[\hom_\A(A,B)^*\cong \hom_\A(B, \SS A),\] where $(-)^*=\hom_k(-,k)$. Now suppose $G$ acts on $\A$ with $|G|$ invertible in $k$. It's shown in \cite{chen}  that $\A^G$ also admits a Serre functor, induced by $\SS$. Here we describe a different proof. In this example, we  write a morphism space $\hom(C, D)$ as $(C, D)$ for short. As usual, we obtain a natural isomorphism \[\sigma_g\colon \SS F_g\overset{\cong}{\lra} F_g\SS\]
from the following  isomorphisms 
\[(B, \SS F_g A)  \cong (F_gA,B)^*\cong (A, F_{g^{-1}}B)^* \cong  (F_{g^{-1}}B, \SS A)\cong  (B, F_g\SS A). \]
  We make our first claim that $\{\sigma_g\}$ makes $\SS$ $G$-equivariant.
  So $\SS$ induces a $k$-linear autoequivalence $\SS^G\colon \A^G\ra\A^G$. Moreover,  we have \[\SS\omega=\omega\SS^G,\quad \sigma=\oplus_{g\in G}\sigma_g\colon\SS^G\ind\overset{\cong}{\ra} \ind\SS.\] For $X, Y\in \A^G$, let \[f_{X,Y}\colon (\ind\,\omega(X),Y)^*\lra (Y,\SS^G\ind\,\omega(X))\] be the composition of the following  isomorphisms
  \[(\ind\,\omega X, Y)^*\ra (\omega X, \omega Y)^*\ra (\omega Y, \SS \omega X)\ra (Y, \ind\SS\omega X)\ra (Y, \SS^G\ind\,\omega X).\]
	 For each $X\in \A^G$,  recall that the composition $X\overset{\eta_X}{\ra} \ind\circ \omega(X)\overset{\epsilon_X}{\ra} X$ of adjunctions  equals $|G|\cdot \id_X$. Let $e_X\colon \ind\circ \omega(X)\ra \ind\circ \omega(X)$ be the idempotent $1-\frac{1}{|G|}\eta_X\epsilon_X$. We make our second claim that there is a commutative diagram
	 \[\xymatrix@C+20pt{(\ind\,\omega(X),Y)^*\ar[r]^{(-\circ  e_X)^*}\ar[d]_{f_{X,Y}} & (\ind\,\omega(X), Y)^*\ar[d]^{f_{X,Y}}\\
	(Y, \SS^G\ind\,\omega(X))\ar[r]^{\SS^G e_X\circ -} & (Y, \SS^G \ind\,\omega(X)),}\]
bifunctorial in $X, Y\in \A^G$. 
Then there is a unique bifunctorial isomorphism \[\kappa_{X,Y}\colon (X,Y)^*\lra (Y, \SS^G X)\] making the following diagram commute 
\[\xymatrix@C+20pt{(X, Y)^*\ar[r]^<<<<<<<<{(-\circ \eta_X)^*} \ar@{-->}[d]^{\kappa_{X,Y}} & (\ind\,\omega(X),Y)^*\ar[r]^{(-\circ  e_X)^*}\ar[d]^{f_{X,Y}} & (\ind\,\omega(X), Y)^*\ar[d]^{f_{X,Y}}\\
	 (Y, \SS^G X)\ar[r]^<<<<<<<<{\SS^G\eta_X\circ -} & (Y, \SS^G\ind\,\omega(X))\ar[r]^{\SS^G e_X\circ -} & (Y, \SS^G \ind\,\omega(X))}\] 
	because the two left horizontal arrows are kernels. 
	This shows  that $\SS^G$ is a Serre functor of $\A^G$.

It remains to prove our two claims. Let us show our first claim. The trianglular commutative diagram for $\SS$ to be $G$-equivariant is easy to be established and to establish the rectangular commutative diagram, it sufficies to prove the commutativity of  the following diagram
\[\begin{footnotesize}\xymatrix@R-10pt{(\SS F_g F_h C, \SS F_g F_h C)\ar[rr]^{\sigma_{g,F_h C}\circ -} \ar[dd] && (\SS F_g F_h C, F_g \SS F_h C)\ar[d]\\
	  && (F_{g^{-1}}\SS F_gF_h C, \SS F_h C)\ar[dd]_{\sigma_{h,C}\circ -} \\
	  (\SS F_g F_h C, \SS F_{hg} C)\ar[dd]^{\sigma_{hg,C}\circ -} && \\
		&& (F_{g^{-1}} \SS F_g F_h C, F_h \SS C)\ar[d]\\
		(\SS F_g F_h C, F_{hg} \SS C)\ar[rr] && (\SS F_g F_h C, F_g F_h \SS C),}\end{footnotesize}
\]
	where the unadorned arrows represent the obvious isomorphisms. Indeed, the commutativity of the above diagram can be obtained by combining a few commutative diagrams,  among which two diagrams  take the form ~\eqref{eq: commutativity} given in Corollary~\ref{commutativity} and the others are evidently commutative. 
	To show our second claim, note that for $(A,\{\lambda_g^A\})\in \A^G$, since $\SS$ is $G$-equivariant, $\SS A$ has an induced equivariant structure $\{\lambda_g^{\SS A}:=\SS(\lambda_g^A)\circ \sigma_{g,A}^{-1}\}$. Given another object $(B, \{\lambda_g^B\})\in \A^G$, one can check that the isomorphism \[\hom_\A(A,B)^*\cong \hom_\A(B,\SS A)\] is in fact an isomorphism of $G$-representations. To do this, one needs the commutativity of the following diagram 
	\[\xymatrix{(A,B)^*\ar[r]\ar[d]^{(F_g)^*} & (B,\SS A)\ar[r]^{F_g} & (F_g B, F_g\SS A)\ar[d]^{\sigma_{g,A}^{-1}\circ -}\\
(F_g A, F_g B)^*\ar[rr] & & (F_g B, \SS F_g A),}\] which can be obtained by combining a few obviously commutative diagrams. 
	And then one readily sees that there is a  commutative diagram 
	\[\begin{small}
\xymatrix@C-15pt{(\ind\,\omega X, Y)^*\ar[r]\ar[d]^{(-\circ e_X')^*} & (\omega X, \omega Y)^*\ar[r]\ar[d]^{a^*} & (\omega Y, \SS \omega X)\ar[r]\ar[d]^{b} & (Y, \ind\SS\omega X)\ar[r]\ar[d]^{c} & (Y, \SS^G\ind\,\omega X)\ar[d]^{\SS^G e_X'\circ -}\\
			 (\ind\,\omega X, Y)^*\ar[r] & (\omega X, \omega Y)^*\ar[r] & (\omega Y, \SS \omega X)\ar[r] & (Y, \ind\SS\omega X)\ar[r] & (Y, \SS^G\ind\,\omega X),}
		 \end{small}
	 \]
	 where 
	\[\begin{aligned}
		e_X' & =1-e_X=\frac{1}{|G|}\eta_X\varepsilon_X,\\
			a & =\{u\mapsto \frac{1}{|G|}\sum_{g\in G}u.g\},\\
			b & =\{v\mapsto \frac{1}{|G|}\sum_{g\in G}v.g\},\\
			c & =\{w\mapsto \sigma_A\circ \SS^G e'_X\circ \sigma_A^{-1}\circ w\}.
	\end{aligned}
\]
We are done.
\end{exm}


\subsection{Trivial action}\label{sec: trivial}

Now let us define a trivial action (see ~\cite[\S3.3]{KP}). Let $G$ be a finite group with unit $e$ and $\A$ an additive category. 
\begin{defn}
	An action $\{F_g,\delta_{g,h}\mid g,h\in G\}$ of $G$ on $\A$ is said to be \emph{trivial} if there is a \emph{trivialization} of this action, i.e., a collection $\{\xi_g\colon F_g\overset{\sim}{\ra} \id_\A\mid g\in G\}$  of natural isomorphisms  such that for each $g,h\in G$, the following diagram 
	\[\xymatrix{F_gF_h\ar[r]^{\delta_{g,h}} \ar[d]_{F_g\xi_h} & F_{hg}\ar[d]^{\xi_{hg}}\\
	F_g \ar[r]^{\xi_g} & \id}\]
	commutes. 
\end{defn}

In this case, $\xi_e$ coincides with the unit of the action. Moreover, each $A\in \A$ has an equivariant structure $\{\xi_{g,A}\mid g\in G\}$, which  is called the \emph{trivial equivariant structure}. $G$ acts trivially on  $\hom_{\A}(A, B)$ if $A,B$ are equipped with the trivial equivariant structures.  If $\A$ is linear over a field $k$ then for $A\in \A$ we have an isomorphism \[\oplus \xi_{g,A}\colon \ind(A)\overset{\cong}{\lra}  A\otimes kG\] in $\A^G$, where the $A$ on the right hand side is equipped with the trivial equivariant structure and $kG$ is the (right) regular  representation  of $G$. 


In case of a trivial action, it is  known that $\A^G$ decomposes as in the following proposition (cf. \cite[\S 4.2]{BKR}, \cite[Proposition 3.3]{KP}).

\begin{prop}\label{trivial action}
	If $G$ acts trivially on an idempotent-complete additive $k$-category $\A$ and $|G|$ is invertible in $k$ then 
	\[\A^G=\coprod_{\rho\in \irr(G)}\A\otimes \rho\simeq \coprod_{\rho\in \irr(G)}\A,\] where for each $A\otimes \rho\in \A\otimes \rho$, $A$ is equipped with the trivial equivariant structure.
\end{prop}

\begin{proof}
	As in \cite[Proposition 3.3]{KP}, one uses Lemma~\ref{tensor hom} to conclude that \[\hom_{\A^G}(\A\otimes\rho, \A\otimes \rho')=0\] for $\rho\neq \rho'\in \irr(G)$ and that the functor  \[ \A\lra \A\otimes \rho,\quad A\mapsto A\otimes \rho \] sending a morphism $f$ in $\A$ to $f\otimes \id_\rho$ is an equivalence for each $\rho\in \irr(G)$. To show $\A^G=\coprod_{\rho\in \irr(G)}\A\otimes \rho$, it remains to show that each $X\in \A^G$ decomposes as $X=\oplus_{\rho\in \irr(G)} X_\rho$, where $X_\rho\in \A\otimes \rho$. We need the following simple  
	\begin{lem}\label{simple lem}
		Suppose $\B$ is an idempotent-complete additive category. Let $Y_1, Y_2\in \B$ with $\hom(Y_1,Y_2)=0=\hom(Y_2, Y_1)$. If $Y\in\B$ is a direct summand of $Y_1\oplus Y_2$ then $Y$ decomposes as $Y= Y_3\oplus Y_4$ such that $Y_3$ resp. $Y_4$ is a direct summand of $Y_1$ resp. $Y_2$.  
	\end{lem}
	Since $X$ is a direct summand of \[\ind \circ \omega(X)\cong  \omega(X)\otimes kG\cong \oplus_{\rho\in \irr(G)}\omega(X)^{\oplus n_\rho}\otimes \rho,\]
	where $n_\rho=\dim_k\rho^\vee$, we know that $X=\oplus_{\rho\in \irr(G)} X_\rho$ by the above lemma, where $X_\rho$ is a direct summand of $\omega(X)^{\oplus n_\rho}\otimes \rho$. Since $\A\otimes \rho\simeq \A$ ($\rho\in \irr(G)$) is idempotent-complete, $\A\otimes \rho$ as a  full subcategory of $\A^G$ is closed under direct summands. So $X_\rho\in \A\otimes \rho$.  
	We are done.
\end{proof}

\begin{rmk}
	We remark that we cannot drop the assumption that $\A$ is idempotent-complete (which is missing in \cite[Proposition 3.3]{KP}). For example, let the bounded derived category $\D^b(\mod k)$ be endowed with the trivial $G$-action induced by the trivial action of $G$ on $k$. Let $\A$ be the smallest triangulated subcategory of $\D^b(\mod k)$ containing $k^{\oplus 2}$. Take two characters $\chi_1\neq \chi_2$ of $G$ over $k$ (suppose the character group $G^*$ is not trivial).  Then the object $k\otimes \chi_1\oplus k\otimes \chi_2\in \A^G$ does not lie in the subcategory $\coprod_{\rho\in \irr(G)} \A\otimes \rho$. 
\end{rmk}

Given a group homomorphism $\phi\colon G_1\ra G_2$ between two finite groups and a $G_2$-action $\{M_a,\alpha_{a,a'}\mid a,a'\in G_2\}$ on $\A$, restriction along $\phi$ yields a $G_1$-action \[\{N_b:=M_{\phi(b)}, \beta_{b,b'}:=\alpha_{\phi(b),\phi(b')}\mid b,b'\in G_1\}\] on $\A$. In particular, consider  a normal subgroup $H$ of a finite group $G$. If $\A$ admits a $G/H$-action then restriction along the quotient map $G\epic G/H$ yields a $G$-action  
 on $\A$, which in turn restricts to a trivial $H$-action  on $\A$.

Now suppose $\A$ is an idempotent complete additive $k$-category with a $G/H$-action, where $k$ is a field, and equip $\A$ with the induced $G$-action. 
We want to relate $\A^{G/H}$ to $\A^G$. Since the action of $H$ on $\A$ is trivial, there is a decomposition \[\A^H=\coprod_{\rho\in \irr(H)} \A\otimes \rho.\] Since the restriction functor  $\res^H_G\colon \A^G\ra \A^H$ is faithful, there is a decomposition \[\A^G=\coprod_{\rho\in \irr(H)} (\A^G)_\rho,\] where  
 $(\A^G)_\rho=(\res^H_G)^{-1}(\A\otimes \rho)$ is  the full subcategory of $\A^G$ consisting of those objects whose images under  $\res^H_G$ lie in $\A\otimes \rho$.  
\begin{prop}\label{trivial2}
	Let $\A$ be an idempotent-complete additive $k$-category with a $G/H$-action.  
	Then  we have an equivalence \[F\colon \A^{G/H}\overset{\simeq}{\lra} (\A^G)_{\rho_0}\]  
	such that \[\pr\circ \ind^G\cong F\circ \ind^{G/H},\] where $\rho_0$ is the trivial $H$-representation  over $k$ and $\pr$ is the projection of $\A^G$ to its component $(\A^G)_{\rho_0}$.

\end{prop}
\begin{proof}

Let $G/H=\{Hg_i \mid 1\leq i\leq n\}$, where we have fixed a complete set of representatives $\{g_1=e,\dots,g_n\}$ and
let $\{F_{Hg_i}, \mu_{Hg_i,Hg_j}\}$ be the $G/H$-action  on $\A$.
Define a functor \[F\colon \A^{G/H}\lra (\A^G)_{\rho_0}\] as follows: $F$ sends  an object $(A, \{\mu_{Hg_i}^A\})$ in $\A^{G/H}$ to $(A, \{\nu_g^{A}\})$, where $\nu_g^{A}=\mu_{Hg_i}^A$ (suppose $g\in Hg_i$); $F$ sends a morphism $f$ in $\A^{G/H}$ to $f$. 
One readily sees that $F$ is well-defined. In particular, we have  $(A, \{\nu_g^{A}\})\in (\A^G)_{\rho_0}$.  It's also easy to see that $F$ is an equivalence. 
Let $\iota\colon (\A^G)_{\rho_0}\monic \A^G$ be the inclusion functor. Note that we have two adjoint pairs  \[(\pr\circ \ind^G, \omega\circ \iota)\quad\text{and}\quad (F\circ\ind^{G/H}, \omega\circ F^{-1}),\]
where $F^{-1}$ is a quasi-inverse of $F$. 
Since $\omega=\omega\circ \iota\circ F$, we have $\pr\circ \ind^G\cong F\circ \ind^{G/H}$.

\end{proof}

\subsection{Reconstruction of $\A$ from $\A^G$ via equivariantization}\label{sec: reconstruction}
Let $k$ be a field.   
Suppose $G$ is a finite group with $|G|$ invertible in $k$ and $\A$ is an idempotent-complete  additive $k$-category with an action of $G$. 
It's well-known that  there is  a de-equivariantization process $(-)_G$ such that $(\A^G)_G\simeq \A$  (see \cite{DGNO}).  We show in this subsection that if moreover $G$ is solvable  and $k$ is algebraically closed then it is possible to obtain $\A$ from $\A^G$ via equivariantization. This is a common generalization  of  Reiten and Riedtman's result  \cite[Propsition 5.3]{RR}  and Elagin's result \cite[Theorem 1.3]{E}.  Our proof  is almost along the same lines with   Elagin's proof of \cite[Theorem 1.3]{E}.

 Denote by $G'$ the commutator group of $G$. 
 Observe that we have a strict action $\{-\otimes \chi\mid \chi\in G^*\}$ of the character group  $G^*$ on $\A^G$. 
 \begin{thm}\label{A^G A}
	 Suppose that $\A$ is an idempotent-complete $k$-linear additive category, where $k$ is an algebraically closed field, and that $|G|$ is invertible in $k$. Then there is an equivalence \[(\A^G)^{G^*}\simeq \A^{G'}.\] In particular, if  $G$ is solvable then we can obtain $\A$ from $\A^G$ after a finite sequence of equivariantization; if  $G$ is abelian then we have an equivalence $(\A^G)^{G^*}\simeq \A$.
\end{thm}
The special case when $G$ is abelian  is Elagin's result.  When $\A=\Mod R$ for a finite dimensional algebra $R$,  our proposition specializes to Reiten and Riedtman's results \cite[Proposition 5.3]{RR} and \cite[Corollary 5.2]{RR}.

\begin{proof}

\def\trivial{\textup{trivial}}
	The adjoint pair \[ \xymatrix{\omega\colon (\A^G)^{G^*} \ar@<0.5ex>[r] & \ar@<0.5ex>[l] \A^G :\ind}\] gives us a  comonad $T_1=(F_1=\omega\circ \ind, \epsilon_1, \Delta_1)$ on $\A^G$;
	the adjoint pair \[\xymatrix{\ind^G_{G'}\colon \A^{G'} \ar@<0.5ex>[r] & \ar@<0.5ex>[l] \A^{G}: \res_G^{G'}}\] gives us another  comonad $T_2=(F_2=\ind^G_{G'}\circ \res_{G}^{G'}, \epsilon_2, \Delta_2)$ on $\A^G$. Since $\omega$ and $\ind^G_{G'}$ are separable functors and since $(\A^G)^{G^*}$ and $\A^{G'}$  are idempotent-complete by Lemma~\ref{adjoint}(4),  we know from  Corollary~\ref{comonadicity2} that the comparison functors \[(\A^G)^{G^*}\lra (\A^G)_{T_1},\quad \A^{G'}\lra (\A^G)_{T_2}\]  are equivalences. To show $(\A^G)^{G^*}\simeq \A^{G'}$, it sufficies to show that the two comonads $T_1, T_2$ on $\A^G$ are isomorphic. 

We have \[F_1=-\otimes (\oplus_{\chi\in G^*}\chi) \colon \A^G\lra \A^G.\] Moreover, $\epsilon_1\colon F_1\ra \id$ is induced by the morphism of $G$-representations \[\oplus_{\chi\in G^*}\chi\lra k_{\trivial},\quad 1_\chi\mapsto \left\{\begin{array}{ll} 1 & \,\,\text{if $\chi= k_{\trivial}$}\\
	0 & \,\,\text{otherwise}\end{array}\right.\] where $k_{\trivial}$ is the trivial representation, and that $\Delta_1\colon F_1\ra F_1\circ F_1$ is induced by the morphism of $G$-representations 
\[\oplus_{\chi\in G^*}\chi\lra (\oplus_{\chi\in G^*}\chi)\otimes (\oplus_{\chi\in G^*}\chi),\quad 1_\chi\mapsto \sum_{\chi'\in G^*} 1_{\chi\chi'^{-1}}\otimes 1_{\chi'}.\]

Suppose $G/G'=\{G'g_i\mid 1\leq i\leq n\}$ ($g_1=e$) and let  $kG/G'$ be the $k$-space with basis $\{G'g_i\mid 1\leq i\leq n\}$ equipped with the natural  $G$-action. 
We have still another comonad $T_3=(F_3=-\otimes kG/G', \eta_3, \Delta_3)$  on $\A^G$, where $\epsilon_3\colon F_3\ra \id$ is defined by the morphism of $G$-representations \[kG/G'\lra k_{\trivial},\quad G'g_i\mapsto  1\] and  $\Delta_3\colon F_3\ra F_3\circ F_3$ is defined by the morphism of $G$-representations \[kG/G'\lra kG/G'\otimes kG/G',\quad G'g_i\mapsto G'g_i\otimes G'g_i.\]

We first show that $T_1\cong T_3$. Since the abelian group $G/G'$ by assumption  has order invertible in the algebraically closed field $k$, we have an isomorphism \[\alpha\colon \oplus_{\chi\in G^*}\chi \lra kG/G'\] as $G$-representations (and as $G/G'$-representations). 
By the linear independence of  characters of $G/G'$, we can suppose $\alpha(\sum_{\chi\in G^*} 1_\chi)=G'e$. 
$\alpha$ induces  a natural  transform $\varepsilon\colon F_1\ra F_3$. To see that $\varepsilon$ is compatible with the counits and comultiplications, it sufficies to note that we have two commutative diagrams of $G$-representations
	 \[\xymatrix@C-20pt{\oplus_{\chi\in G^*}\chi \ar[rr] \ar[d]_\alpha && k_{\trivial}\ar@{=}[d] & & \oplus_{\chi\in G^*}\chi \ar[rr] \ar[d]_\alpha && (\oplus_{\chi\in G^*}\chi)\otimes(\oplus_{\chi\in G^*}\chi)  \ar[d]^{\alpha\otimes \alpha}\\
	 kG/G' \ar[rr] && k_{\trivial}, & & kG/G' \ar[rr] && kG/G'\otimes kG/G'.}\]
This allows us to conclude $T_1\cong T_3$. 

Now we show $T_3\cong T_2$. Define a natural transform $\theta\colon F_2\ra F_3$ by $\theta_{(A,\{\lambda_g^A\})}=\oplus_{i=1}^n\lambda_{g_i}^A$
for  $(A,\{\lambda_g^A\})\in \A^G$. $\theta$ is well-defined and is obviously  an isomorphism. It's direct to check that $\theta$ is compatible with the counits and the comultiplications and thus yields an isomorphism $T_3\cong T_2$ of comonads. 

In conclusion, we have $T_1\cong T_3\cong T_2$. We are done.

\end{proof}

\begin{rmk}
	We remark here that  \cite[Proposition 5.1, Corollary 5.2]{RR} hold in greater generality. In fact, if $k$ is an algebraically closed field, $A$ a $k$-algebra and $G$ a finite group with $|G|$ invertible in $k$ then there is an algebra isomorphism 
	\[
		\begin{aligned}
			\phi\colon (AG)G^*\quad & \lra & &\End_{AG'}(AG),\\ 
	 ag\chi  \quad & \mapsto & &\{bh\mapsto \chi(h)agbh,\,\, \text{for}\,b\in A,h\in G\},\,\,\text{for}\, a\in A, g\in G, \chi\in G^*,
 \end{aligned}
 \] 
where $AG$ is considered as a right $AG'$-module; in particular, the skew group ring $(AG)G^*$ is Morita equivalent to the skew group ring  $AG'$. All  arguments in \cite{RR} work in the current setting except the arguments showing the surjectiveness of $\phi$. Actually, one can use the linear independence of characters of $G$ to conclude that $\phi$ is surjective.

\end{rmk}

\section{Equivariantization of triangulated categories}

Throughout this section, $G$ denotes a finite group. 
\subsection{(Pre-)triangulated structure on categories of equivariant objects}\label{sec: tri}

Let $\D$ be a triangulated (resp. pretriangulated) category with a $G$-action. 
With an expectation that the category $\D^G$ of equivariant objects is also triangulated (resp. pretriangulated), it's natural to consider an action that is compatible with the triangulated (resp. pretriangulated) structure. 

\begin{defn}\label{admissible tri}
	A  $G$-action $\{F_g, \delta_{g,h}\}$  on $\D$ is said to be \emph{admissible} if   $F_g$'s are  exact autoequivalences with commutating isomorphisms $\theta_g\colon F_g\circ [1]\ra [1]\circ F_g$ such that $\{\theta_g^{-1}\}$ makes the translation functor $[1]$ of $\D$ $G$-equivariant.
\end{defn}

In this case, the $G$-equivariant translation functor $[1]$ of $\D$ induces a translation functor (i.e. an autoequivalence) $\D^G\ra \D^G$, which is also denoted by $[1]$. We have an equality $\omega\circ [1]=[1]\circ \omega$ and a natural isomorphism \[\theta=\oplus_{g\in G}\theta_g\colon \ind\circ [1]\overset{\sim}{\ra} [1]\circ \ind.\]

\begin{defn} Given an admissible $G$-action on a triangulated (resp. pretriangulated) category $\D$, a triangulated (resp. pretriangulated) structure on $\D^G$ such that the forgetful functor $\omega\colon \D^G\ra \D$ is exact is called a \emph{canonical triangulated} (resp. \emph{pretriangulated}) \emph{structure}.
\end{defn}

By virtue of  Balmer's theorem \cite[Theorem 4.1]{Balmer}, Chen observed in \cite{XWC} that under certain assumptions there exists a unique canonical pretriangulated structure on $\D^G$. Here we slightly  generalize his result \cite[Lemma 4.4(3)]{XWC}. 

\begin{prop}\label{unique}\label{pretri}
Suppose $G$ acts admissibly  on $\D$ with $|G|$  invertible in $\D$. Then there exists a unique canonical pretriangulated structure on $\D^G$ and exact triangles in $\D^G$ are precisely those triangles whose images under $\omega$ are exact in $\D$. 
\end{prop}

To prove this, the following fact is basic for us.
\begin{lem}[{\cite{BS}}]\label{BS}
	
\begin{enum2}
\item Let $\T$ be a category and let us be given a commutative diagram in $\T$ \[\xymatrix{A\ar[r]^a \ar[d]^{e_1} & B\ar[r]^b \ar[d]^e & C\ar[d]^{e_3}\\
A \ar[r]^a & B\ar[r]^b & C.}\]
If $e_1, e_3$ are idempotents and $b$ is a pseudo-cokernel of $a$ then there is an idempotent $e_2\colon B\ra B$ making the diagram commute with $e_2$ in lieu of $e$.

\item Let $\T$ be a pretriangulated category and let \[\Delta_1=(X,Y,Z,u,v,w), \quad\Delta_2=(X', Y', Z', u', v', w')\] be two triangles in $\T$. Then the direct sum \[\Delta_1\oplus\Delta_2=(X\oplus X', Y\oplus Y', Z\oplus Z', u\oplus u', v\oplus v', w\oplus w')\] is an exact triangle in $\T$ iff so are $\Delta_1$ and $\Delta_2$.

\item  
	Let $\T$ be a triangulated (resp. pretriangulated)  category. Then its idempotent-completion $\widehat{\T}$ admits a unique triangulated (resp. pretriangulated) structure such that the canonical functor $\iota\colon \T\ra \widehat{\T}$ is exact. More precisely, an exact triangle in $\widehat{\T}$ is a triangle $(T_1,T_2,T_3,u,v,w)$ for which there is a triangle  $(T_1',T_2',T_3',u',v',w')$ such that the triangle \[(T_1\oplus T_1', T_2\oplus T_2', T_3\oplus T_3', u\oplus u',v\oplus v',w\oplus w')\] in $\widehat{\T}$ is isomorphic to a triangle $((U,1),(V,1),(W,1),r,s,t)$  for some exact triangle $(U,V,W,r,s,t)$ in $\T$.

\end{enum2}
\end{lem}

\begin{proof}
	(1) follows from the proof of \cite[Lemma 1.13]{BS}. (2) is \cite[Lemma 1.6]{BS}.  The assertion in (3) on a pretriangulated category follows from the proof of \cite[Theorem 1.5]{BS} and
 the assertion in (3)  on a triangulated category is  \cite[Theorem 1.5]{BS} (the nontrivial part is the verification of (TR4), the octahedron axiom). 

\end{proof}

\begin{proof}[Proof of Proposition~\ref{unique}]
 
	Let $\Sigma$ be the set of triangles in $\D^G$ whose images under $\omega$ are exact triangles in $\D$. 

	{\it Step 1}.   
	We show that $\Sigma$ defines a pretriangulated structure on $\D^G$ under an additional assumption that $\D$ is idempotent-complete\footnote{When the $G$-action is strict, Chen obtained this assertion (i.e., \cite[Lemma 4.4(3)]{XWC}) by applying Balmer's theorem \cite[Theorem 4.1]{Balmer}.  We believe that Balmer's theorem could be adapted to imply our assertion. More precisely, some related commutativity  in  the strict sense required in \cite{Balmer} could be weakened to cover our situation. Anyway, we give a direct proof of our assertion here and our proof essentially follows the proof of \cite[Theorem 4.1]{Balmer}.}. 
In this case, we can use the full strength of the adjoint triple $(\ind, \omega, \ind)$.

We first prove  (TR3). Let us be given a commutative diagram with rows being triangles in $\Sigma$  
\[\xymatrix{X\ar[r]^a\ar[d]^u &Y\ar[r]^b \ar[d]^v & Z\ar[r]^c & X[1]\ar[d]^{u[1]}\\
X'\ar[r]^{a'} & Y'\ar[r]^{b'}  & Z'\ar[r]^{c'} & X'[1].}\]
We have a commutative diagram in $\D$ with rows being exact triangles
\[\xymatrix{\omega(X)\ar[r]^{\omega(a)}\ar[d]^{\omega(u)} & \omega(Y)\ar[r]^{\omega(b)} \ar[d]^{\omega(v)} & \omega(Z)\ar[r]^{\omega(c)}\ar[d]^{r} & \omega(X)[1]\ar[d]^{\omega(u)[1]}\\
\omega(X')\ar[r]^{\omega(a')} & \omega(Y')\ar[r]^{\omega(b')}  &\omega(Z')\ar[r]^{\omega(c')} & \omega(X')[1].}\]
Define $w=\frac{1}{|G|}\sum_{g\in G}r.g$. Then there is a commutative diagram in $\D^G$ \[\xymatrix{X\ar[r]^a\ar[d]^u & Y\ar[r]^b \ar[d]^v & Z\ar[r]^c\ar[d]^w & W\ar[d]^{u}\\
X'\ar[r]^{a'} & Y'\ar[r]^{b'}  & Z'\ar[r]^{c'} & W',}\]
This proves (TR3).

All conditions in (TR1) and (TR2) are ready to be verified except that  each morphism $a\colon X_1\ra X_2$ embeds into a triangle in $\Sigma$, which we now show. 
Let \[\omega(X_1)\overset{\omega(a)}{\lra} \omega(X_2)\overset{d}{\lra} A\overset{f}{\lra}\omega(X_1)[1]\] be an exact triangle in $\D$.  Let $\eta\colon \id\ra \ind \circ \omega$ resp. $\epsilon\colon \ind\circ \omega \ra \id$ be the unit resp. counit of the adjoint pair $(\omega, \ind)$ resp. $(\ind, \omega)$. There is a commutative diagram with rows being triangles in $\Sigma$
 \[\xymatrix@C+10pt{ \ind\circ \omega(X_1)\ar[r]^{\ind\circ \omega(a)}\ar[d]^{e_1} & \ind\circ \omega(X_2) \ar[r]^{\ind(d)}\ar[d]^{e_2} & \ind(A) \ar[r]^{\theta_{\omega(X_1)}\circ\ind(f)}\ar[d]^{e_3} & \ind\circ \omega(X_1)[1]\ar[d]^{e_1[1]}\\
 \ind\circ \omega(X_1)\ar[r]^{\ind\circ \omega(a)}  &  \ind\circ \omega(X_2) \ar[r]^{\ind(d)} & \ind(A) \ar[r]^{\theta_{\omega(X)}\circ\ind(f)} & \ind\circ \omega(X)[1].}\]
 where $e_i$ ($i\in \{1,2\}$) is the idempotent $1-\frac{1}{|G|}\eta_{X_i}\epsilon_{X_i}$ and the existence of $e_3$ follows from (TR3). Since $\ind$ as a left adjoint preserves pseudo-cokernel, $\ind(d)$ is a pseudo-cokernel of $\ind\circ \omega(a)$. So we can take $e_3$ to be an idempotent by Lemma~\ref{BS}(1). 
Moreover, we have $\epsilon_{X_i}=\coker(e_i)$ ($i\in \{1,2\}$) and $u=\coker(e_3)$ exists since $\D^G$ is idempotent-complete.  Then there is a morphism of triangles
\[\xymatrix@C-12pt{\Delta \ar[d]_{\gamma_1} & \ind\circ \omega(X_1)\ar[rr]^{\ind\circ \omega(a)} \ar[d]_{\epsilon_{X_1}} && \ind\circ \omega(X_2) \ar[rr]^{\ind(d)}\ar[d]_{\epsilon_{X_2}} && \ind(A) \ar[rr]^{\theta_{\omega(X_1)}\circ\ind(f)} \ar[d]_u && \ind\circ \omega(X_1)[1]\ar[d]_{\epsilon_{X_1}[1]}\\
\Delta_1 & X_1\ar[rr]^a && X_2\ar[rr]^b && X_3\ar[rr]^c && X_1[1],}\]
where $b$ and $c$ are the unique morphisms making the above diagram commutative.
Furthermore, the commutative diagram 
 \[\xymatrix@C+10pt{ \ind\circ \omega(X_1)\ar[r]^{\ind\circ \omega(a)}\ar[d]^{1-e_1} & \ind\circ \omega(X_2) \ar[r]^{\ind(d)}\ar[d]^{1-e_2} & \ind(A) \ar[r]^{\theta_{\omega(X_1)}\circ\ind(f)}\ar[d]^{1-e_3} & \ind\circ \omega(X_1)[1]\ar[d]^{1-e_1[1]}\\
 \ind\circ \omega(X_1)\ar[r]^{\ind\circ \omega(a)}  &  \ind\circ \omega(X_2) \ar[r]^{\ind(d)} & \ind(A) \ar[r]^{\theta_{\omega(X)}\circ\ind(f)} & \ind\circ \omega(X)[1]}\]
 yields a morphism of triangles 
 \[\xymatrix@C-12pt{\Delta \ar[d]_{\gamma_2} & \ind\circ \omega(X_1)\ar[rr]^{\ind\circ \omega(a)} \ar[d]_{f_1} && \ind\circ \omega(X_2) \ar[rr]^{\ind(d)}\ar[d]_{f_2} && \ind(A) \ar[rr]^{\theta_{\omega(X_1)}\circ\ind(f)} \ar[d]_{f_3} && \ind\circ \omega(X_1)[1]\ar[d]_{f_1[1]}\\
 \Delta_2 & X_1'\ar[rr]^{a'} && X_2'\ar[rr]^{b'} && X_3'\ar[rr]^{c'} && X_1'[1],}\]
 where $f_i=\coker(1-e_i)$ and $a',b',c'$ are the unique morphisms making the  above diagram commutative. Then we have an isomorphism of triangles \[(\gamma_1,\gamma_2)^t\colon \Delta\lra \Delta_1\oplus \Delta_2.\] 
 Now that $\omega(\Delta)$ is an exact triangle in $\D$, so is the direct summand $\omega(\Delta_1)$ by Lemma~\ref{BS}(2). Hence $\Delta_1$ lies in $\Sigma$. This shows that each morphism in $\D^G$ embeds into a triangle in $\Sigma$.

{\it Step 2}. We show that $\Sigma$  defines a pretriangulated structure on $\D^G$  without assuming that $\D$ is idempotent-complete. Let $\widehat{\D}$ be the idempotent-completion of $\D$.  By Lemma~\ref{BS}(3), $\widehat{\D}$ admits a unique pretriangulated structure such that the canonical functor $\iota\colon \D\ra \widehat{\D}$ is exact. By Step 1, the set of triangles in  $\widehat{\D}^G$ whose images under $\omega$ are exact triangles in $\widehat{\D}$ defines a pretriangulated structure such that the forgetful functor $\omega\colon \widehat{\D}^G\ra \widehat{\D}$ is exact. Now that the essential image of the fully faithful functor $\iota^G\colon \D^G\ra \widehat{\D}^G$ induced by the canonical functor $\iota\colon \D\ra \widehat{\D}$ is closed under shifts and taking cone,  $\D^G$ admits a unique pretriangulated structure such that $\iota^G\colon \D^G\ra \widehat{\D}^G$ is exact and the set of exact triangles in $\D^G$ is precisely $\Sigma$. Hence $\Sigma$ defines a pretriangulated structure on $\D^G$. 

{\it Step 3}. We show the uniqueness of the desired pretriangulated structure on $\D^G$. This follows from the well-known fact that if two sets $\Sigma$ and $\Sigma'$ of triangles satisfying $\Sigma'\subset \Sigma$ defines a pretriangulated structure on an additive category then $\Sigma'=\Sigma$.

\end{proof}
\begin{rmk}
The condition that $|G|$ is invertible in $\D$ is necessary. For example,  let $k$ be a field whose characteristic divides $|G|$.  The semisimple abelian category $\mod k$  admits a triangulated structure, while the abelian category $(\mod k)^G\simeq \mod kG$ does not admit a pretriangulated structure since it is not semisimple.
\end{rmk}

As an immediate corollary, we have the following fact on  the uniqueness of a canonical triangulated structure under the assumption of its existence.  
\begin{cor}\label{uniquetri}
	If a triangulated category $\D$ has an admissible $G$-action and $|G|$ is invertible in $\D$ then a canonical triangulated structure  on $\D^G$ is unique once it exists and exact triangles consists of those triangles in $\D^G$ whose images are exact in $\D$.
\end{cor}



Here are two examples of existence of a canonical triangulated structure.
\begin{exm}
	Suppose $G$ acts trivially on a triangulated category $\D$ linear over a field $k$ and $|G|$ is invertible in $k$. Then  $\D^G$ admits a canonical triangulated structure\footnote{This trivial action case was considered in \cite[Proposition 3.3]{KP} but the proof there is flawed.}.  Indeed, let the idempotent-completion $\widehat{\D}$  of $\D$ be equipped with the induced action, which is also a trivial action. Equip $\widehat{\D}^G$  with the canonical pretriangulated structure. Then the equivalence $\widehat{\D}^G\simeq \coprod_{\rho\in \irr(G)} \widehat{\D}$ given in  Proposition~\ref{trivial action} is an exact equivalence. Thus  the canonical pretriangulated structure on $\widehat{\D}^G$ is actually a  triangulated structure.   
	And then we can see that $\D^G$ admits a canonical triangulated structure.
\end{exm}

\begin{exm}
	Suppose $G$ acts admissibly on a triangulated category $\D$ with $|G|$ invertible in $\D$ and suppose $\D^G$ admits a canonical triangulated structure. Then $\widehat{\D}^G$ also admits a canonical triangulated structure and there is an exact equivalence $F\colon \widehat{\D^G}\ra \widehat{\D}^G$, where $F$ is the equivalence defined in the proof of Lemma~\ref{idem}. It sufficies to note that $F$ is exact in the current setup. Indeed, we have $F\circ \iota=\iota^G$, where $\iota^G\colon \D^G\ra \widehat{\D}^G$ is as in Example~\ref{exidem} and $\iota\colon \D^G\ra \widehat{\D^G}$ is the canonical functor. Then one readily concludes from Lemma~\ref{BS} that $F$ is exact. 
\end{exm}

 In general,  we don't know the existence of a canonical triangulated structure. In the remaining of this subsection as well as the next subsection, we  are going to give several instances of existence of a canonical triangulated structure.  We will show that certain formations of a new triangulated category almost commute  with equivariantization. In each case to be considered, we have a natural comparison functor. 
We will prove that such a comparison functor induces an  equivalence between idempotent-completions and thus the functor itself is an  equivalence up to retracts. This will imply the existence of a canonical triangulated structure on the  category of equivariant objects under concern. And then the corresponding comparison functor becomes an exact equivalence up to retracts. This is exactly what we mean by saying ``almost commute''.
 
We start with localization. 
See \cite[Corollary 4.4]{CCZ} for an analogue of the following fact for abelian categories.

\begin{thm}\label{localize}
Suppose $\D$ is a triangulated category with an admissible $G$-action,  $|G|$ is invertible in $\D$, and $\D^G$ admits a canonical triangulated structure.
Let $\C$ be a $G$-invariant triangulated subcategory of $\D$. Then:
\begin{enum2}
\item  The Verdier quotient $\D/\C$ carries an induced admissible $G$-action and $(\D/\C)^G$ admits a unique canonical triangulated structure. 

\item There is a natural exact functor $\D^G/\C^G\ra (\D/\C)^G$ that is  an equivalence up to retracts. In particular, we have  an exact equivalence $\D^G/\C^G\simeq (\D/\C)^G$ if $\D^G/\C^G$ is idempotent-complete.
\end{enum2}
\end{thm}
To prove this, we need the following two simple lemmas. 
\begin{lem}[{See e.g. \cite[Lemma 1.1]{orlov3}}]
	\label{localize adjoint}
	Suppose $F: \A\rightleftarrows \B: H$ is an adjoint pair of exact functors between triangulated categories with unit $\eta$ and counit $\epsilon$. Suppose $\A_1, \B_1$ are respectively triangulated subcategories of $\A$ and $\B$ such that $F(\A_1)\subset \B_1, H(\B_1)\subset \A_1$. Then the adjoint pair $(F,H,\eta,\epsilon)$ induces an adjoint pair $\underline{F}: \A/\A_1\rightleftarrows \B/\B_1: \underline{H} $ with unit $\underline{\eta}$ and counit $\underline{\epsilon}$ such that $\underline{\eta}_A=\eta_A/1$ and $\underline{\epsilon}_B=\epsilon_B/1$. Moreover, if $F$ is a separable functor then so is $\underline{F}$. 
 \end{lem}

 \begin{lem}
	 Let $F: \A\rightleftarrows \B: H$ be an adjoint pair between additive categories with unit $\eta$ and counit $\epsilon$. Then the adjoint pair $(F,H,\eta,\epsilon)$ induces an adjoint pair $\widehat{F}: \widehat{\A}\rightleftarrows \widehat{\B}: \widehat{H}$ between the idempotent-completions with unit $\widehat{\eta}$ and counit $\widehat{\epsilon}$ such that $\widehat{\eta}_{(A,e_A)}=\eta_A\circ e_A$ and $\widehat{\epsilon}_{(B,e_B)}=e_B\circ \epsilon_B$. Moreover, if $F$ is a separable functor then so is $\widehat{F}$.
\end{lem}

\begin{proof}[Proof of Theorem~\ref{localize}]
Obviously $\C^G$ admits a unique triangulated structure such that the inclusion $\C^G\ra \D^G$ is exact, which  in fact coincides with the canonical triangulated structure. 
The admissible action of $G$ on $\D$ induces an admissible $G$-action on $\D/\C$ in an obvious way. If we endow $(\D/\C)^G$ with the canonical pretriangulated structure then the forgetful funtor $\omega\colon (\D/\C)^G\ra \D/\C$ is exact. 
 The Verdier quotient functor $Q\colon \D\ra\D/\C$ is evidently $G$-equivariant  and then we have an exact functor $Q^G\colon \D^G\ra (\D/\C)^G$. Since $\C^G$ lies in the kernel of $Q^G$,
 $Q^G$ induces an exact functor $F\colon \D^G/\C^G\ra (\D/\C)^G$, which further extends to an exact functor 
 \[\xymatrix{\widehat{F}\colon \widehat{\D^G/\C^G}\ar[r] & \widehat{(\D/\C)^G}}\]
 between the idempotent completions. Moreover, the adjoint pair $\omega\colon \D^G \rightleftarrows \D : \ind$  of exact functors induces an adjoint pair $\underline{\omega}\colon \D^G/\C^G  \ra \D/\C: \underline{\ind},$ which in turn induces an adjoint pair \[\xymatrix{\underline{\widehat{\omega}}\colon \widehat{\D^G/\C^G} \ar@<0.5ex>[r] & \ar@<0.5ex>[l] \widehat{\D/\C}: \underline{\widehat{\ind}};}\]
 on the other hand, the adjoint pair $\omega\colon (\D/\C)^G\rightleftarrows \D/\C: \ind$ induces an adjoint pair \[\xymatrix{\widehat{\omega}: \widehat{(\D/\C)^G}  \ar@<0.5ex>[r] & \ar@<0.5ex>[l] \widehat{\D/\C}: \widehat{\ind}.}\] Then we have two  comonads on $\widehat{\D/\C}$ \[T_1=(F_1=\widehat{\underline{\omega}}\circ \widehat{\underline{\ind}}, \epsilon_1, \Delta_1),\quad T_2=(F_2=\widehat{\omega}\circ \widehat{\ind}, \epsilon_2, \Delta_2).\] One readily sees $T_1=T_2$, which we redenote by $T$. 
 Since $\widehat{\underline{\omega}}$ and $\widehat{\omega}$ are separable functors, by Corollary~\ref{comonadicity2}, we have equivalences \[K_1\colon \widehat{\D^G/\C^G}\overset{\simeq}{\lra} (\widehat{\D/\C})_T,\quad K_2\colon \widehat{(\D/\C)^G}\overset{\simeq}{\lra} (\widehat{\D/\C})_T,\]  where $K_i$'s are the comparison functors. Since $K_2\circ \widehat{F}=K_1$, $\widehat{F}$ is an equivalence. 
 So the pretriangulated structure on $\widehat{(\D/\C)^G}$ is a triangulated structure. It follows that the canonical pretriangulated structure on $(\D/\C)^G$ is a triangulated structure. This finishes the proof. 

\end{proof}

A standard construction of a triangulated category is the stabilization of a Frobenius category, due  to  Heller \cite{heller} and  Happel \cite{happel}. First let us recall  that an \emph{exact category} $(\E, \Sigma)$ (or simply $\E$) is an additive category $\E$ with an exact structure, that is, a collection $\Sigma$ of kernel-cokernel pairs closed under isomorphism and satisfying a system of axioms in the sense of Quillen \cite{quillen}, where a \emph{kernel-cokernel pair} in $\E$ means a pair $X\overset{i}{\ra} Y\overset{d}{\ra} Z$ of composable morphisms in $\E$ for which $i$ is the kernel of $d$ and $d$ is the cokernel of $i$. If $X\overset{i}{\ra} Y\overset{d}{\ra} Z$ lies in $\Sigma$ then it is called a \emph{conflation} in $\E$, $i$  an \emph{inflation} and $d$ a \emph{deflation}.  Here we will follow the following simplified yet equivalent system of axioms given by Keller in \cite[Appendix A]{Keller}:
\setlist[description]{font=\mdseries, style=multiline, labelwidth=4.5em, labelsep=-1em, leftmargin=4pc}
\renewcommand{\descriptionlabel}[1]{\hspace\labelsep \upshape #1}
\begin{description}
\item[$(\Ex0)$] $1_0$ is a deflation.
	\item[$(\Ex1)$] A composition of two deflations is a deflation.
	\item[$(\Ex2)$] For each morphism $f\colon Z'\ra Z$ and each deflation $d\colon Y\ra Z$ in $\A$, there is a cartesian square with $d'$ a deflation \[\xymatrix{Y'\ar[r]^{d'} \ar[d]_{f'} & Z'\ar[d]^f\\ Y\ar[r]^d & Z.}\] 
	\item[$(\Ex2)^{\op}$] For each morphism $h\colon X\ra X'$  and each inflation $i\colon X\ra Y$ in $\E$, there is a co-cartesian square with $i'$ an inflation \[\xymatrix{X\ar[r]^i\ar[d]_h & Y\ar[d]^{h'} \\X'\ar[r]^{i'} & Y'.}\]
\end{description}
For example,  an additive category  admits an exact structure consisting of   split short exact sequences; an abelian category  admits an exact structure consisting of all short exact sequences.

An object $P$ in an exact category $\E$ is called  \emph{projective}  if for each deflation $d\colon X\ra Y$ in $\E$ and each morphism $f\colon P\ra Y$, there exists $u\colon P\ra X$ such that $f=du$.   $\E$ is said to  \emph{have enough projectives} if for each $E\in \E$, there is a deflation $P\ra E$ with $P\in\Proj\E$.  Dually one defines the concept of an \emph{injective} object and \emph{having enough injectives}. We denote by $\Proj \E$ resp. $\Inj \E$ the full subcategory of projective resp. injective objects in $\E$. 
An exact category $\E$ is called a \emph{Frobenius category} if it  has enough projectives and enough injectives and if projectives coincide with injectives.

An additive functor $F\colon \E_1\ra \E_2$ between two exact categories is called \emph{exact} if it preserves exact structures, that is, $F$ maps conflations in $\E_1$ to conflations in $\E_2$. 
An action $\{F_g, \delta_{g,h}\}$ of $G$ on the exact category $\E$ is called \emph{admissible} if $F_g$'s are exact autoequivalences. 
We show that given an admissible action, $\E^G$ admits an induced exact structure.

\begin{lem}\label{equiv exact}
	Let $\E$ be an exact category with an admissible $G$-action. 
	There is a unique maximal  exact structure on $\E^G$ such that the forgetful functor $\omega\colon \E^G\ra \E$ is exact. If $|G|$ is invertible in $\E$ then  we have \[\Proj \E^G=(\Proj \E)^G,\quad \Inj \E^G=(\Inj \E)^G.\] 
\end{lem}
\begin{proof}
Define $\Omega$ be the collection of  pairs of composable morphisms in $\E^G$ whose images under $\omega$ are conflations in $\E$. One readily sees that $\Omega$ consists of kernel-cokernel pairs and is closed under isomorphism.  We show that $\Omega$ defines an exact structure on $\E^G$. Evidently, the axioms $(\Ex0)$ and $(\Ex1)$  hold for $\Omega$. 
	We check $(\Ex2)$. For each $f\in \hom_{\E^G}(Z',Z)$ and each deflation $d\in \hom_{\E^G}(Y,Z)$, there is a cartesian square 
	\[\xymatrix{U\ar[r]^{d'} \ar[d]_{f'} & \omega(Z') \ar[d]^{\omega(f)}\\
	\omega(Y)\ar[r]^{\omega(d)} & \omega(Z),}\]
	where $d'$ is a deflation in $\E$. 
	One easily sees that $U$ admits an equivariant structure $\{\lambda_g^U\}$ such that $Y'=(U, \{\lambda_g^U\}) \in \A^G$ fits into a cartesian square
	\[\xymatrix{Y'\ar[r]^{d'} \ar[d]_{f'} & Z' \ar[d]^f\\
	Y\ar[r]^{d} & Z,}\]
	and $d'$ is a deflation in $\E^G$. Similarly one checks $(\Ex2)^{\op}$. Hence $\Omega$ defines an exact structure on  $\E^G$. Obviously, if another collection $\Omega'$ of kernel-cokernel pairs defines an exact structure on $\E^G$ such that $\omega\colon \E^G\ra \E$ is exact then $\Omega'\subset \Omega$. Thus $\Omega$ is the unique maximal exact structure with the desired property.

Suppose $|G|$ is invertible in $\E$. We want to show $\Proj \E^G=(\Proj \E)^G$ and similar arguments allow one to conclude  $ \Inj \E^G= (\Inj \E)^G.$  Obviously  $\Proj\E$ is closed under the action of $G$ and so the expression $(\Proj\E)^G$ makes sense. We first show $(\Proj\E)^G\subset \Proj\E^G$. 
	Let $P\in (\Proj\E)^G$. 
	We have $\ind\circ \omega(P)\in \Proj\E^G$. 
	Since the composition  $P\ra \ind\circ \omega(P)\ra P$ of adjunctions is $|G|\cdot \id_P$, $P$ is projective by ~\cite[Corollary 11.4]{Theo}. 
This shows  $(\Proj\E)^G\subset \Proj\E^G$.
Now we show the converse inclusion. 
Suppose $P\in \Proj\E^G$. 
	Let $v\colon Y\ra Z$ be a deflation in $\E$ and $u\colon \omega(P)\ra Z $ be a morphism in $\E$. Denote by $\tau$ the counit of the adjoint pair $(\omega, \ind)$. Let $h$ be the image of $u$ under the adjunction map $\hom(\omega(P), Z)\cong \hom(P, \ind(Z))$. 
	Then we have a commutative diagram 
	\[\xymatrix{& \omega(P)\ar[d]^{\omega(h)} \ar[dl]_{\omega(f)}\ar@/^2.5pc/[dd]^u\\
		\omega\circ \ind(Y)\ar[r]_{\omega\circ\ind(v)} \ar[d]_{\tau_Y} & \omega\circ\ind(Z)\ar[d]^{\tau_Z}\\
	Y\ar[r]^v & Z,}\]
		where the existence of $f$ follows from the fact that $P$ is projective in $\E^G$.  Then we have $u=vw$, where $w=\tau_Y\circ \omega(f)$. Thus $\omega(P)\in \Proj\E$. This shows $\Proj\E^G\subset (\Proj\E)^G$. Hence  $\Proj\E^G= (\Proj\E)^G$.

\end{proof}

\begin{cor}\label{Fro}
	Let $\E$ be a Frobenius category with an admissible $G$-action and suppose $|G|$ is invertible in $\E$.  Endow $\E^G$ with the  exact structure  asserted in Lemma~\ref{equiv exact} and suppose further that $\E^G$ has enough projectives and injectives (say, when $\E$ is idempotent-complete). Then $\E^G$ is a Frobenius category. 
	
\end{cor}
\begin{proof}
		We first show that if $\E$ is idempotent-complete then $\E^G$  contains enough projectives and injectives. Let $X\in \E^G$. The adjunction $\ind \circ\omega(X)\ra X$ is a split epimorphism and thus is a deflation since $\E$ is idempotent-complete.  Moreover,  we have a deflation $f\colon M \ra \omega(X)$ in $\E$ with $M$ projective. Then the composition of deflations \[\ind(M)\overset{\ind(f)}{\lra} \ind\circ\omega(X)\lra X\] yields a deflation $\ind(M)\ra X$. 
	This allows us to conclude that $\E^G$ has enough projectives  since $\ind(M)$ is projective in $\E^G$. Similar arguments apply to conclude that $\E^G$ has enough injectives. 

Now we show $\E^G$ is Frobenius under the given assumptions.  By Lemma~\ref{equiv exact},  we have $\Proj \E^G=\Inj \E^G$. So $\E^G$ is a Frobenius category provided that $\E^G$ contains enough projectives and injectives. 

\end{proof}

Let $\E$ be a Frobenius category. Let us recall from \cite{happel} the definiton of  the stable category $\underline{\E}$ of $\E$. $\underline{\E}$ shares the same objects with $\E$. For $X, Y\in \underline{\E}$, \[\hom_{\underline{\E}}(X,Y)\coloneq \hom_{\E}(X,Y)/\I(X,Y),\] where $\I(X,Y)$ is the subgroup consisting of those morphisms $f\in \hom_\E(X,Y)$ such that $f$ factors through an injective object in $\E$. 
$\underline{\E}$ is a triangulated category with translation functor $\SSS$ such that for each $X\in \underline{\E}$, $\SSS X$ fits into a chosen conflation $X\ra \I X\ra \SSS X$ in $\E$, where $\I X$ is injective in $\E$. 
An exact triangle in $\underline{\E}$ is a triangle in $\underline{\E}$ that is isomorphic to a triangle  $X\overset{\bar{u}}{\ra} Y\overset{\bar{v}}{\ra} Z\overset{-\bar{w}}{\ra} \SSS X$ in $\underline{\E}$ 
for which there is a commutative diagram in $\E$
\[\xymatrix{X\ar[r]^u\ar@{=}[d] & Y\ar[r]^v\ar[d] & Z \ar[d]^{w}\\
X\ar[r] & \I X \ar[r] & \SSS X,}\]
where the upper row is a conflation in $\E$ and the lower row is the chosen conflation  to define $\SSS$.

\begin{thm}\label{stable}	
	With the same assumptions as in Corollary~\ref{Fro}, we have:	
	\begin{enum2} 
	\item $\underline{\E}$ carries an induced admissible $G$-action and $\underline{\E}^G$ admits a unique canonical triangulated structure. 
	
	\item There is a natural  exact functor $\underline{\E^G}\ra \underline{\E}^G$ that is always an exact equivalence up to retracts and that is an exact equivalence if $\underline{\E^G}$ is idempotent-complete. 
	\end{enum2}

\end{thm}
The proof of this theorem is in spirit much like the proof of Theorem~\ref{localize} whereas there are some different ingredients. 
In particular, we need  recall a lemma on a $\partial$-functor in the sense of Keller \cite{keller2}.  Recall that given  a Frobenius category $\E$ and  a pretriangulated category $\D$, a $\partial$-functor $F\colon \E\ra \D$ is an additive functor $F$ such that for any conflation \[\nabla\colon X\overset{u}{\lra} Y \overset{v}{\lra} Z\] in $\E$, there is a morphism $\alpha_\nabla\colon F(Z)\ra F(X)[1]$ making 
\[F(X)\overset{F(u)}{\lra} F(Y) \overset{F(v)}{\lra} F(Z) \overset{\alpha_\nabla}{\lra} F(X)[1]\]
an exact triangle in $\D$, functorial in the  sense that given a morphism 
\[\xymatrix@C-8pt{\nabla \ar[d] &  X\ar[rr]^u\ar[d]^r && Y \ar[rr]^v \ar[d]^s &&  Z\ar[d]^t\\
	\nabla' &  X'\ar[rr]^{u'} &&  Y' \ar[rr]^{v'} && Z'}\]
of conflations, the following diagram commutes
\[\xymatrix{F(X)\ar[r]^{F(u)}\ar[d]^{F(r)} &  F(Y) \ar[r]^{F(v)} \ar[d]^{F(s)} &  F(Z)\ar[d]^{F(t)}\ar[r]^{\alpha_\nabla} & F(X)[1]\ar[d]^{F(r)[1]}\\
F(X')\ar[r]^{F(u')} &  F(Y') \ar[r]^{F(v')} &  F(Z') \ar[r]^{\alpha_{\nabla'}} & F(X')[1].}\]

\begin{lem}[{See e.g. \cite[Lemma 2.5]{chen2}}]
	\label{partial exact}
	If $F\colon \E\ra \D$ is a $\partial$-functor, where $\E$ is a Frobenius category and $\D$ is a pretriangulated category, and $\Inj \E\subset \ker F$ then $F$ induces an exact functor $\underline{F}\colon \underline{\E}\ra \D$. Consequently, if $H\colon \E\ra \E'$ is an exact fuctor between Frobenius categories preserving injectives then $H$ induces an exact functor $\underline{H}\colon \underline{\E}\ra \underline{\E'}$. 
\end{lem}

We also need the following two simple lemmas, whose proofs are left to the reader.
\begin{lem}\label{partial}
	Let $\E$  and $\E^G$ be  as assumed  in Corollary~\ref{Fro}. Let $\D$ be a pretriangulated category with an admissible $G$-action and suppose $|G|$ is invertible  in $\D$.
 If $F\colon \E\ra \D$ is a $G$-equivariant $\partial$-functor then the induced functor $F^G\colon \E^G\ra \D^G$ is also a $\partial$-functor, where $\D^G$ is equipped with the canonical pretriangulated structure.
\end{lem}

	\begin{lem}
		Let $F: \E_1\rightleftarrows \E_2: H$ be an adjoint pair of exact functors between Frobenius categories. Suppose further $F$ and $H$ preserve injective objects. Then the adjoint pair  $(F, H)$ induces an adjoint pair $\underline{F}:\underline{\E_1}\rightleftarrows \underline{\E_2}: \underline{H}$ of exact functors  between the stable categories. Moreover, if $F$ is separable then so is $\underline{F}$. 
	\end{lem}

\begin{proof}[Proof of Theorem~\ref{stable}]

It's easy to see that an admissible $G$-action on $\E$ induces an admissible $G$-action on $\underline{\E}$. 
Moreover, the canonical functor $H\colon \E\ra \underline{\E}$ is  $G$-equivariant and  the induced functor $H^G\colon \E^G\ra \underline{\E}^G$ is a  $\partial$-functor by Lemma~\ref{partial}.
	Since $\Inj\E^G\subset \ker H^G$, $H^G$ induces an exact functor $\underline{H^G}\colon \underline{\E^G}\ra \underline{\E}^G$ by Lemma~\ref{partial exact}. We claim that $\underline{H^G}$ is an equivalence up to retracts. 
Indeed, the adjoint pair $\omega: \E^G\rightleftarrows \E: \ind$ induces an adjoint pair $\underline{\omega}: \underline{\E^G}\rightleftarrows \underline{\E}: \underline{\ind}$ between the stable categories, which in turn induces an adjoint pair \[\xymatrix{\widehat{\underline{\omega}}: \widehat{\underline{\E^G}} \ar@<0.5ex>[r] & \ar@<0.5ex>[l] \widehat{\underline{\E}}: \widehat{\underline{\ind}}}\] between the idempotent-completions. 
This gives us a  comonad  \[T_1=(\underline{\widehat{\omega}}\circ \underline{\widehat{\ind}}, \epsilon_1, \Delta_1)\] on $\underline{\widehat{\E}}$. 
		Moreover, the adjoint pair $\omega: \underline{\E}^G\rightleftarrows \underline{\E}: \ind$ induces an adjoint pair \[\xymatrix{\widehat{\omega}: \widehat{\underline{\E}^G} \ar@<0.5ex>[r] & \ar@<0.5ex>[l] \widehat{\underline{\E}}: \widehat{\ind}}\] between the idempotent-completions, 
		which gives us a comonad \[T_2=(\widehat{\omega}\circ \widehat{\ind}, \epsilon_2, \Delta_2).\]  We have $T_1= T_2$ and the composition  \[\widehat{\underline{\E^G}}\simeq \widehat{\underline{\E}}_{T_1}= \widehat{\underline{\E}}_{T_2}\simeq \widehat{\underline{\E}^G}\] gives us an equivalence which is isomorphic to the exact functor 
		\[\widehat{\underline{H^G}}\colon \widehat{\underline{\E^G}}\lra \widehat{\underline{\E}^G}\] induced by $\underline{H^G}: \underline{\E^G}\ra \underline{\E}^G$. So $\underline{H^G}$ is an equivalence up to retracts, as claimed. Moreover, as argued in Theorem~\ref{localize}, the canonical pretriangulated structure on $\underline{\E}^G$ is actually a triangulated structure. This finishes our proof.

\end{proof}

\begin{rmk}
	Sonsa~\cite{Sonsa} proposed a solution to a \emph{triangulated} equivariant category as follows. Suppose $\B$ is a pretriangulated DG category with a $G$-action and $\T$ the homotopy category $H^0(\B)$ of $\B$. By definition, $\T^G_\B$ is the homotopy category $H^0((\B^G)^{\textup{pretr}})$ of the pretriangulated hull $(\B^G)^{\textup{pretr}}$ of the DG category $\B^G$ (see \cite[Definition 3.9]{Sonsa}). His approach does not directly deal with the problem whether there is a triangulated structure on $\T^G$ (in our notation, not in Sonsa's notation). Moreover, in general, we don't know whether there is an equivalence $\T^G_\B\ra \T^G$. 
	For example, if $\A$  is an additive category with a $G$-action and  $\B=\C_{dg}(\A)$ is the DG category of differential complexes over $\A$ equipped with the induced $G$-action then $\K(\A)_\B^G$ is triangle isomorphic to $\K(\A^G)$. But from   Exmaple~\ref{exm: homotopy}, we know that generally we only have an equivalence up to retracts $\K(\A^G)\ra \K(\A)^G$. 
\end{rmk}

\subsection{Examples of canonically triangulated categories of equivariant objects} \label{sec: exm}
We shall apply the previous two propositions to  homotopy categories, derived categories and singularity categories.  

Let $*\in \{\emptyset, +, -, b\}$. For an additive category $\A$, we denote by $\C^*(\A)$ the category of arbitrary (resp. bounded below,  bounded above,  bounded) differential complexes over $\A$ if $*=\emptyset$ (resp. $+$,  $-$,  $b$) and  by $\K^*(\B)$ the corresponding homotopy category.

\begin{exm}\label{exm: homotopy}
	Let $\A$ be an additive category with a $G$-action and suppose $|G|$ is invertible in $\A$. Then $\K^*(\A)$ carries an induced admissible $G$-action, $\K^*(\A)^G$ admits a canonical triangulated structure and we have an exact equivalence up to retracts \[\K^*(\A^G)\lra \K^*(\A)^G,\]
which is an exact equivalence if $\K^*(\A^G)$ is idempotent-complete. 
	
Indeed, recall that a pair $X^\bullet\overset{u}{\ra} Y^\bullet\overset{v}{\ra} Z^\bullet$ of composable morphisms in $\C^*(\A)$ is called \emph{termwise split} if each term $X^n\overset{u^n}{\ra} Y^n\overset{v^n}{\ra} Z^n$ is a split short exact sequence. It's well-known (see \cite{Kel}) that the set of all such termwise split pairs defines an exact structure on $\C^*(\A)$,  $\C^*(\A)$ is then a Frobenius category whose projective objects are precisely those null-homotopic complexes and $\K^*(\A)$ is the corresponding stable category of $\C^*(\A)$. 
   The $G$-action  on $\A$ induces an admissible $G$-action on the Frobenius category $\C^*(\A)$. 
 There is an obvious isomorphism $\C^*(\A^G)\cong \C^*(\A)^G$ of exact categories. So there is an exact isomorphism \[\K^*(\A^G)=\underline{\C^*(\A^G)}\cong \underline{\C^*(\A)^G}.\]
 By Theorem~\ref{stable}, $\K^*(\A)^G$ admits a canonical triangulated structure and there is an exact equivalence up to retracts $\underline{\C^*(\A)^G}\ra \K^*(\A)^G$. Hence we have an exact equivalence up to retracts \[\K^*(\A^G)\lra \K^*(\A)^G.\] 
 
 For example, if $\A$ is idempotent-complete, in which case $\K^b(\A^G)$ is idempotent-complete by \cite[Theorem 2.9]{IY}, we have an exact equivalence \[\K^b(\A^G)\overset{\simeq}{\lra} \K^b(\A)^G.\] In particular, we have an exact equivalence \[\K^b(\Proj \B^G)=\K^b((\Proj \B)^G)\overset{\simeq}{\lra} \K^b(\Proj \B)^G\] provided that $G$ acts on  an idempotent-complete exact category $\B$  with $|G|$  invertible in $\B$.
\end{exm}

For an abelian category $\A$, we denote by $\D^*(\A):=\K^*(\A)/\K^*_{ac}(\A)$ the corresponding derived category of $\A$, where $\K_{ac}^*(\A)$ is the full subcategory of $\K^*(\A)$ consisting of acyclic complexes. 
\begin{exm}\label{der}

 Let $G$ act on an abelian category $\A$ with $|G|$ invertible in $\A$.  Then $\D^*(\A)$ carries an induced admissible $G$-action,  $\D^*(\A)^G$ admits a canonical triangulated structure
 and there is an exact equivalence up to retracts \[\D^*(\A^G)\lra \D^*(\A)^G,\] which is an exact equivalence if $\D^*(\A^G)$ is idempotent-complete.  

 Indeed, we can identify $\K^*(\A^G)$ canonically as a strictly full triangulated subcategory of $\K^*(\A)^G$ by Example~\ref{exm: homotopy}  and meanwhile $\K^*_{ac}(\A^G)$ can be identified with a strictly full triangulated subcategory of $\K^*_{ac}(\A)^G$. 
 Note that  $\ind\circ \omega(U)\in\K^*_{ac}(\A)^G$ actually lies in $\K^*_{ac}(\A^G)$ for $U\in \K^*_{ac}(\A)^G$ and that each morphism $f\colon X\ra Y$ in $\K^*(\A)^G$ factors through $\ind\circ \omega(Y)\in \K^*_{ac}(\A^G)$, where $X\in \K^*(\A^G)$ and $Y\in \K^*_{ac}(\A)^G$.  Thus the induced exact functor \[F\colon \K^*(\A^G)/\K^*_{ac}(\A^G)\lra \K^*(\A)^G/\K^*_{ac}(\A)^G\] is fully faithful by \cite[\S 2, Th\'eor\`eme 4-2]{verdier}, which is moreover dense up to retracts.
 By Theorem~\ref{localize}, we have an exact equivalence up to retracts 
 \[H\colon \K^*(\A)^G/\K^*_{ac}(\A)^G\lra (\K^*(\A)/\K^*_{ac}(\A))^G.\]
 The composition  yields an exact equivalence up to retracts \[HF\colon \D^*(\A^G)\lra \D^*(\A)^G.\] 

 For example, since $\D^b(\A^G)$ is idempotent-complete by \cite[Corollary 2.10]{BS} or \cite[Theorem]{LC},  we have an exact equivalence $\D^b(\A^G)\simeq \D^b(\A)^G$. For another example,  if $\A$ satifies AB $4)$ (resp. AB $4^*)$)\footnote{Recall that AB $4)$ means that $\A$ has arbitrary direct coproducts and direct coproducts are exact. AB $4^*)$ is the dual condition on $\A$.} in the sense of Grothendieck~\cite{Tohoku}, 
 then we have an exact equivalence $\D(\A^G)\simeq \D(\A)^G$. Indeed, in this case, $\A^G$ also satisfies AB $4)$ (resp. AB $4^*)$). 
 Then  $\D(\A^G)$ has arbitrary direct coproducts (resp. direct products) by \cite[Lemma 1.5]{BoNe} and thus is idempotent-complete by \cite[Proposition 1.6.8, Remark 1.6.9]{Neeman}.

That there is an equivalence $\D^b(\A^G)\simeq \D^b(\A)^G$ is already shown in \cite{XWC, E} (the equivalence is moreover shown to be an exact equivalence in \cite{XWC}). And the result for idempotent-complete unbounded derived categories is also mentioned in \cite[Remark 4.2]{XWC}. 
\end{exm}

\begin{exm}

	Let $\A$ be an abelian category with enough projectives. Recall that  an object $M\in \A$ is called a  \emph{Gorenstein projective} object if there is an exact sequence of projectives
	\[P^\bullet\colon \quad \dots \lra P^{-1}\overset{d^{-1}}{\lra} P^0\overset{d^0}{\lra} P^1\lra \dots\]  such that $\hom_{\A}(P^\bullet, \Proj \A)$ are exact and $M\cong \im d^{-1}$. 
The full subcategory  $\GP(\A)$ of $\A$ consisting of Gorenstein-projectives is a Frobenius category whose exact structure is given by those short exact sequences in $\A$ with terms lying in $\GP(\A)$ and we have $\Proj \GP(\A)=\Proj\A$.  Denote
\[\D^b_{\sg}(\A):=\D^b(\A)/\K^b(\Proj \A). \] This yields  Orlov's triangulated categories of singularities  \cite{orlov3} for a (two-sided) notherian graded algebra $R$ over a field $k$  when $\A$ is taken to be the category of finitely generated right modules over $R$ and the category of finitely generated graded right modules over $R$.
A generalization of Buchweitz's theorem \cite[Theorem 4.4.1]{buch} and Happel's theorem \cite[Theorem 4.6]{happel2}  
states that  the exact functor \[F_\A\colon \underline{\GP(\A)}\lra \D^b_{\sg}(\A),\] induced by the canonical inclusions $\GP(\A)\monic \A\monic \D^b(\A)$, is fully faithful, and $F_\A$ is an equivalence if each object in $\A$ has finite Gorenstein projective dimension. 
For the above statements,  one can see Zhang's book \cite{zhang}, Chapter 8 if one is able to read Chinese. 
The most general form of Buchweitz's theorem and Happel's theorem is due to Chen \cite{chen2}, who introduced relative singularity categories as a  generalization of Orlov's singularity categories for a noetherian graded algebra. 

Now suppose $\A$ has a $G$-action and $|G|$ is invertible in $\A$.  $\GP(\A)$ is obviously $G$-invariant and is equipped with an admissible $G$-action.  Thus $\underline{\GP(\A)}$ is equipped with an admissible $G$-action by Proposition~\ref{stable}. It's also evident that $F_\A$ is $G$-equivariant and thus we have an induced  functor \[F_\A^G\colon \underline{\GP(\A)}^G\lra \D^b_{\sg}(\A)^G.\]
We have $\Proj \A^G=(\Proj \A)^G$ by Lemma~\ref{equiv exact} and one readily  sees $\GP(\A^G)=\GP(\A)^G$, where one needs to use the fact that $\GP(\A)$ is closed under direct summands to conclude $\GP(\A)^G\subset \GP(\A^G)$. 
Moreover, Since the composition of the canonical inclusions $\GP(\A)\monic \A\monic \D^b(\A)$ is a  $G$-equivariant $\partial$-functor, the induced functor $\GP(\A)^G\ra \D^b(\A)^G$
as a $\partial$-functor (by Lemma~\ref{partial})  induces an exact functor \[H\colon \underline{\GP(\A)^G}\lra \D^b(\A)^G/\K^b(\Proj \A)^G.\]
 Then one sees that there is a commutative  diagram 
 \[\xymatrix{\underline{\GP(\A^G)}\ar@{=}[r] \ar[d]^{F_{\A^G}} & \underline{\GP(\A)^G}\ar[r] \ar[d]^H & \underline{\GP(\A)}^G\ar[d]^{F_\A^G} \\
 \D^b_{\sg}(\A^G)\ar[r]^<<<<{\simeq} & \D^b(\A)^G/\K^b(\Proj \A)^G\ar[r] & \D^b_{\sg}(\A)^G,}\] where the existence of the left lower horizontal equivalence follows from Example~\ref{exm: homotopy} and Example~\ref{der},  the right upper resp. lower horizontal arrow represents the functor asserted in Theorem~\ref{stable} resp. Theorem~\ref{localize}. One can  obtain a similar result involving relative singularity categories by starting from the setup of \cite{chen2}.

 In particular, consider a positively graded  Iwanaga-Gorenstein ring $R=\oplus_{i\in \NN}R_i$, that is, a positively graded  notherian  ring $R$ such that $R$ has finite injective dimension as a graded left  $R$-module and as a graded right  $R$-module.
 Suppose $R$ has a right $G$-action preserving degree and $|G|$ is invertible in $R$. Clearly the skew group ring $RG$ is  also a positively graded  Iwanaga-Gorenstein ring. 
 Denote by $\modZ R$ the category of finitely generated graded right $R$-modules with homogeneous maps of degree zero and by  $\MCM^\ZZ R$ resp. $\projZ R$ its full subcategory consisting of  Gorenstein-projective resp. projective  modules.
 By adapting classical arguments for a usual(=ungraded) Iwanaga-Gorenstein ring, one can show that \[\MCM^\ZZ R=\{M\in \modZ R \mid \ext^i_{\modZ R}(M, \projZ R)=0, \forall i>0\}\] and then that each object in $\modZ R$ has a finite Gorenstein projective dimension.
 Denote \[\D_{\sg}^{\gr}(R)=\D^b(\modZ R)/\K^b(\projZ R).\]
 Then there is a commutative  diagram 
 \[\xymatrix{\underline{\MCM^{\ZZ} RG} \ar[d]^{\simeq} \ar[r] & \underline{\MCM^\ZZ R}^G \ar[d]^{\simeq} \\
 \D_{\sg}^{\gr}(RG)\ar[r]  & \D_{\sg}^{\gr}(R)^G.}\]
 The horizontal arrows are exact equivalences if  additionally $R_0$ has finite global dimension. Indeed, in this case, $R_0G$ has finite global dimension and by the arguments in \cite[\S5]{BurSte}, 
 $\D^b(\modZ RG)$ admits a weak semi-orthogonal decomposition  $\pair{\D_1, \D_2, \D_3}$   with an exact equivalence $\D_2\simeq \D^{\gr}_{\sg}(RG)$. It follows that $\D_{\sg}^{\gr}(RG)$ is idempotent-complete. 
 See \S\ref{sec: t-str} for the definition of a weak semi-orthogonal decomposition. 
\end{exm}

In the remaining of this section, we will work with triangulated categories $\D$ satisfying the following
\[\begin{aligned}
		\text{\bf (Hypothesis)} &  \quad \text{$G$ acts admissibly on $\D$ with $|G|$ invertible in $\D$ and}\\
	& \quad\text{$\D^G$ admits a canonical triangulated structure.}
\end{aligned}\]

\subsection{T-structure and semi-orthogonal decomposition}\label{sec: t-str}
T-structures are introduced by Beilinson, Bernstein and Deligne  in \cite{BBD}. 
\begin{defn}
	A \emph{t-structure}  on a triangulated category  $\D$ is a pair $(\D^{\leq 0}, \D^{\geq 0})$ of strictly full subcategories of $\D$  satisfying the following axioms: (denoting $\D^{\leq n}=\D^{\leq 0}[-n], \D^{\geq n}=\D^{\geq 0}[-n]$)  
	\begin{enumi}
	\item $\D^{\leq 0}$ (resp. $\D^{\geq 0}$) is closed under suspension (resp. desuspension), 
	\item $\hom_\D(\D^{\leq 0}, \D^{\geq 1})=0$,
	\item each object $A$ in $\D$ fits into an exact triangle $X\ra A\ra Y\ra X[1] $ with $X\in \D^{\leq 0}, Y\in \D^{\geq 1}$.
	\end{enumi}

	A t-structure $(\D^{\leq 0}, \D^{\geq 0})$ is called \emph{nondegenerate} if $\cap_{n\in \ZZ} \D^{\leq n}=0=\cap_{n\in\ZZ}\D^{\geq n}$ and is called \emph{bounded} if each $X\in \D$ lies in some $\D^{[m,n]}:=\D^{\geq m}\cap \D^{\leq n}$. 
	A t-structure $(\D^{\leq 0}, \D^{\geq 0})$  is called \emph{a stable t-structure} if $\D^{\leq 0}$ (equivalently, $\D^{\geq 0}$) is a triangulated subcategory of $\D$. 
\end{defn}
	
\begin{prop}\label{t-str}
	Suppose $\D$ is a triangulated category satisfying {\bf (Hypothesis)}. 
	Let $(\D^{\leq 0},\D^{\geq 0})$ be a $G$-invariant t-structure on $\D$, i.e., a t-structure such that $\D^{\leq 0}$ is $G$-invariant.  
	Then $((\D^{\leq 0})^G, (\D^{\geq 0})^G)$ is a  t-structure on $\D^G$.
 If  $(\D^{\leq 0},\D^{\geq 0})$ is nondegenerate (resp.  bounded)   then so is $((\D^{\leq 0})^G, (\D^{\geq 0})^G)$. 
If  $(\D^{\leq 0},\D^{\geq 0})$  is a stable t-structure then we have a stable t-structure $((\D^{\leq 0})^G, (\D^{\geq 0})^G)$.
\end{prop}

\begin{proof}
	Let $\{F_g, \delta_{g,h}\}$ be the admissible $G$-action on $\D$ and let $\theta_g\colon F_g\circ [1]\ra [1]\circ F_g$ be the commutating isomorphism for $F_g$.
	It's well-known that \[ \D^{\geq 1}=\{X\in \D\mid \hom(\D^{\leq 0},X)=0\}, \quad \D^{\leq 0}=\{X\in \D\mid \hom(X, \D^{\geq 1})=0\}.\]
So $\D^{\leq 0}$ is $G$-invariant iff $\D^{\geq 0}$ is $G$-invariant iff $(F_g(\D^{\leq 0}), F_g(\D^{\geq 0}))=(\D^{\leq 0},\D^{\geq 0})$ for all $g\in G$. In this case, $\D^{\geq m}$, $\D^{\leq n}$ and $\D^{[m,n]}$ are $G$-invariant and we have $(\D^{\geq m})^G[i]=(\D^{\geq m-i})^G, (\D^{\leq n})^G[i]=(\D^{\leq n-i})^G$. 
	
It's easy to see that $(\D^{\leq 0})^G$ (resp. $(\D^{\geq 0})^G$) is closed under suspension (resp. desuspension) and that \[\hom_{\D^G}((\D^{\leq 0})^G, (\D^{\geq 1})^G)=0.\] It remains to prove that each $X\in \D^G$ fits into an exact triangle $A\ra X\ra B\ra A[1]$ with $A\in (\D^{\leq 0})^G, B\in (\D^{\geq 1})^G$.
Let $(E,\{\lambda_g^E\})\in \D^G$ and \[ E_1\overset{u}{\lra} E\overset{v}{\lra} E_2\overset{w}{\lra} E_1[1]\] be an exact triangle with  $E_1\in \D^{\leq 0}, E_2\in \D^{\geq 1}$.
	We will show that the equivariant structure on $E$ induces equivariant structures on $E_1$ and $E_2$. 
	Since $(\D^{\leq 0},\D^{\geq 0})$ is $G$-invariant, for each $g\in G$, we have $F_g(E_1)\in \D^{\leq 0}$ and thus 
	\[\hom_\D^{-1}(F_g(E_1),E_2)=0=\hom_\D(F_g(E_1), E_2).\] Then by \cite[Proposition 1.1.9]{BBD}, there exist unique  \[\lambda^{E_1}_g\colon F_g(E_1)\lra E_1\quad\text{and}\quad \lambda^{E_2}_g\colon F_g(E_2)\lra E_2\] producing a morphism of exact triangles 
	\[\xymatrix{F_g(E_1)\ar[r]^{F_g(u)}\ar@{-->}[d]^{\lambda_g^{E_1}} & F_g(E)\ar[r]^{F_g(v)}\ar[d]^{\lambda_g^E} & F_g(E_2)\ar[r]^{\theta_{g,E_1}\circ F_g(w)}\ar@{-->}[d]^{\lambda_g^{E_2}} & F_g(E_1)[1]\ar@{-->}[d]^{\lambda_g^{E_1}[1]}\\
	E_1\ar[r]^{u} & E \ar[r]^v & E_2 \ar[r]^w & E_1[1].}
\]
	Moreover, for $g,g'\in G$ we have 
	\[
		\begin{aligned}
			\lambda_{g'g}^E\circ F_{g'g}(u) & = \lambda_g^E\circ F_g(\lambda_{g'}^E)\circ \delta_{g,g',E}^{-1}\circ F_{g'g}(u)\\
			& = \lambda_g^E\circ F_g(\lambda_{g'}^E)\circ F_g F_{g'}(u)\circ \delta_{g,g',E_1}^{-1}\\
			& = \lambda_g^E\circ F_g(u)\circ F_g(\lambda_{g'}^{E_1})\circ \delta_{g,g',E_1}^{-1}\\
			& = u\circ \lambda_g^{E_1}\circ F_g(\lambda_{g'}^{E_1})\circ \delta_{g,g',E_1}^{-1},
		\end{aligned}
		\]
		By the uniqueness of $\lambda_{g'g}^{E_1}$, we have \[\lambda_{g'g}^{E_1}=\lambda_g^{E_1}\circ F_g(\lambda_{g'}^{E_1})\circ \delta_{g,g',E_1}^{-1}.\]
		Hence $(E_1,\{\lambda_{g}^{E_1}\})\in (\D^{\leq 0})^G$. Similarly, one sees  
		$(E_2, \{\lambda_g^{E_2}\})\in (\D^{\geq 1})^G$.   Then the exact triangle in $\D^G$
		\[(E_1,\{\lambda_{g}^{E_1}\}) \overset{u}{\lra} (E,\{\lambda_{g}^{E}\}) \overset{v}{\lra} (E_2,\{\lambda_{g}^{E_2}\}) \overset{w}{\lra} (E_1,\{\lambda_{g}^{E_1}\})[1]\]
allows us to conclude that $((\D^{\leq 0})^G,(\D^{\geq 0})^G)$ is a t-structure in $\D^G$.

The assertion on nondegeneracy and boundedness is obvious. If $\D^{\leq 0}$ is a triangulated subcategory of $\D$ then $(\D^{\leq 0})^G$ admits a unique triangulated structure, which  actually coincides with the canonical triangulated structure, such that the inclusion $(\D^{\leq 0})^G\monic \D^G$ is exact. 
Thus we have a stable t-structure $((\D^{\leq 0})^G, (\D^{\geq 0})^G)$. 

\end{proof}

Let $\D$ be a triangulated category. For two full subcategories $\X,\Y$ of $\D$, denote \[\X*\Y=\{A\mid \exists\, \text{exact triangle $X\ra A\ra Y\ra X[1]$ in $\D$ with $X\in\X$, $Y\in \Y$}\}.\] By the octahedron axiom, $*$ is associative. An \emph{admissible subcategory}~\cite{bondal} of $\D$ is a triangulated subcategory $\C$ such that the inclusion $\C\monic \D$ admits a left and a right adjoint. We continue to recall the definition of a (weak) semi-orthogonal decomposition in the  sense of Bondal and Kapranov \cite{BK} and Orlov \cite{orlov3}. 

\begin{defn}
	An $n$-tuple $\pair{\D_1,\dots, \D_n}$ of strictly full triangulated subcategories of  $\D$ is called a \emph{weak semi-orthogonal decomposition} if $(\D_1*\dots*\D_i, \D_{i+1}*\dots*\D_n)$ is a stable t-structure on $\D$ for each $i$ and is called a \emph{semi-orthogonal decomposition} if it is a weak semi-orthogonal decomposition and if additionally  each $\D_i$ is an admissible subcategory of $\D$. 
\end{defn}
One readily sees that our definitions above are equivalent to the origininal ones given in~\cite{BK} and~\cite{orlov3}. Here is a generalization of  \cite[Theorem 3.2]{KP}.
\begin{prop}
	Let $\D$ be a triangulated category satisfying {\bf (Hypothesis)}. If $\D$ admits a (weak) semi-orthogonal decomposition $\pair{\D_1,\dots, \D_n}$, where each $\D_i$ is $G$-invariant, then $\D^G$ admits a (weak) semi-orthogonal decomposition $\pair{\D_1^G, \dots, \D_n^G}$. 
\end{prop}
\begin{proof}
	It suffices to prove the assertion for $n=2$. If $\pair{\D_1,\D_2}$ is a weak semi-orthogonal decomposition of $\D$ then by Proposition~\ref{t-str},  $(\D_1^G,\D_2^G)$ is a stable t-structure on $\D$, in other words, $\pair{\D_1^G, \D_2^G}$ is a weak semi-orthogonal decomposition of $\D$. If moreover $\pair{\D_1, \D_2}$ is a semi-orthogonal decomposition of $\D$ then  by Lemma~\ref{equiv adjoint}, the left resp.  right adjoint of the inclusion $\D_i\monic \D$ induces the left resp. adjoint of the inclusion $\D_i^G\monic \D^G$, and thus $(\D_1^G, \D_2^G)$ is a semi-orthogonal decomposition. 
\end{proof}
\subsection{Exceptional collection and tilting object}\label{sec: tilting}
The following simple observation is needed shortly.  Recall that a collection $\X$ of objects in a triangulated category $\D$ is said to \emph{classically generate} $\D$ if the thick closure of $\X$ coincides with $\D$.
\begin{lem}\label{ind gen}
	Suppose $\D$ is an idempotent-complete triangulated category satisfying {\bf (Hypothesis)}. If $T$ is a classical generator for $\D$ then $\ind(T)$ is a classical generator for $\D^G$.
\end{lem}

\begin{proof}
	Since $\D^G$ is idempotent-complete,  the exact functor $\ind\colon \D\ra \D^G$ is dense up to direct summands. It follows that the thick closure of $\ind(T)$ is the whole $\D^G$.
\end{proof}
Recall from \cite{bondal} that a sequence $(E_1,\dots, E_n)$ in a triangulated $k$-category $\D$, where $k$ is a field, is called an \emph{exceptional collection} if \[\left\{\begin{array}{ll} \hom^{\neq 0}_\D(E_i,E_i)=0, \hom_\D(E_i,E_i)=k & \quad\text{for each $i$,}\\
	\hom^l_\D(E_j,E_i)=0 & \quad\text{for $j>i$ and any $l\in\ZZ$.}\end{array}\right.\] It is called \emph{full} if $(E_1,\dots, E_n)$ classically generates $\D$. Generalizing  {\cite[Theorem 2.1]{E2}}, we show that an exceptional collection in $\D$ consisting of equivariant objects can yield an exceptional collection in $\D^G$. 
\begin{prop}\label{exceptional sequence}
	Suppose $\D$ is an idempotent-complete triangulated $k$-category satisfying	{\bf (Hypothesis)}. 
Suppose $\irr(G)=\{\rho_1, \dots, \rho_m\}$.
	If $\D$ admits an exceptional collection $(E_1,\dots, E_n)$, where each $E_i$ is $G$-equivariant, then $\D^G$ admits an exceptional collection
	\[\left(\begin{matrix} 
		E_1\otimes \rho_1 & \dots & E_n\otimes \rho_1\\
		\vdots & \vdots & \vdots\\
		E_1\otimes \rho_m & \dots & E_n\otimes \rho_m
\end{matrix}\right),\]
	where the order between items in  any fixed column can be arbitrary. 
	If $(E_1,\dots, E_n)$ is full then the above exceptional collection is also full.
\end{prop}
\begin{proof}
	In light of Example~\ref{equiv skew group}, we have  \[\add \ind(E_i)=(\add E_i)^G=\add \bigoplus_{\rho_j\in \irr(G)} E_i\otimes \rho_j,\]  where each $E_i\otimes \rho_j$ is exceptional,  and \[\hom_{\D^G}^l(E_r\otimes \rho_i, E_s\otimes \rho_j)=0 \quad \text{for}\,\, r>s, l\in \ZZ\,\,\text{or}\,\,r=s, i\neq j, l\in\ZZ.\]
			Then  the given collection is clearly  an exceptional collection  and  the assertion on fullness follows from Lemma~\ref{ind gen}.
\end{proof}

Recall that an object $T$ in a triangulated category $\D$ is called \emph{tilting} if $\hom^{\neq 0}_\D(T,T)=0$ and $T$ classically generates $\D$. A $G$-invariant tilting object in $\D$ can yield a tilting object in $\D^G$.
\begin{prop}\label{tilting}
	Suppose $\D$ is an idempotent-complete triangulated category satisfying	{\bf (Hypothesis)}.  If $T$ is a $G$-invariant tilting object in $\D$ then $\ind(T)$ is a tilting object in $\D^G$ whose endomorphism ring is Morita equivalent to the skew group ring  $\End_\D(\omega\circ \ind(T))G$. Moreover, we have \[\gldim \End_{\D^G}(\ind(T))=\gldim \End_{\D}(T),\] where $\gldim R$ denotes the right global dimension of a ring $R$. 
\end{prop}
\begin{proof}
	Since $T$ is $G$-invariant, we have $\omega\circ\ind(T)\simeq T^{\oplus |G|}$. Since $T$ is a tilting object, \[\hom^{\neq 0}_\D(\omega\circ\ind(T),\omega\circ \ind(T))=0,\] hence \[\hom^{\neq 0}_{\D^G}(\ind(T),\ind(T))=0.\] Combining with Lemma~\ref{ind gen}, we know that $\ind(T)$ is a tilting object. Since \[\ind\circ \omega\circ \ind(T)\simeq \ind(T)^{\oplus |G|},\]  $\End_{\D^G}(\ind(T))$ is Morita equivalent to $\End_{\D^G}(\ind\circ \omega\circ \ind(T))$ and thus is Morita equivalent to $\End_{\D}(\omega\circ\ind(T))G$ (see Example~\ref{equiv skew group}). Then  we have 
\[
	\begin{aligned}
		\gldim \End_{\D^G}(\ind(T)) & = \gldim \End_\D(\omega\circ \ind(T))G\\
		& =\gldim \End_\D(\omega\circ \ind(T))\\
		& =\gldim \End_\D(T).
	\end{aligned}
\]
\end{proof}

\begin{exm}
	Let $\PP^d_k$ be the $d$-dimensional projective space  over an algebraically closed field $k$ and $G$ be a finite subgroup of $PSL_{d+1}(k)\cong \aut_k(\PP^d_k)^{\op}$ with $|G|$ invertible in $k$. It's well-known that $(\O, \dots, \O(d))$ is a full exceptional collection and $T:=\oplus_{i=0}^d\O(i)$ is a tilting bundle in $\D^b(\coh\PP^d_k)$. 
	Since each $\O(i)$ ($i\in\ZZ$) is $G$-invariant, by Proposition~\ref{tilting}, $\D^b(\coh^G\PP^d_k)\simeq \D^b(\coh\PP^d_k)^G$ contains a tilting bundle $\ind^G(T)$. 
It's not obvious how direct summands of $\ind^G(T)$ look like.   Let us consider the natural action of  $\overline{G}$ on $\PP^d_k$, where $\overline{G}$ is the preimage  of $G$ in $SL_{d+1}(k)$. In this case, each $\O(i)$ is $\overline{G}$-equivariant. 
We have \[\ind^{\overline{G}}(T)=\add\bigoplus_{\rho\in\irr(\overline{G}), 0\leq i\leq d}  \O(i)\otimes \rho\] and  \[\O(i)\otimes \rho,\quad 0\leq i\leq d, \rho\in \irr(\overline{G})\] can be ordered to form an exceptional collection by Proposition~\ref{exceptional sequence}. 
	 Then by Proposition~\ref{trivial2}, we see that indecomposable direct summands of $\ind^G(T)$ can be ordered to form an exceptional collection. 

When $d=1$, this example is explored in \cite{kirrilov}. 
\end{exm}

\subsection{Dimension and saturatedness}\label{sec: dim}

Rouquier introduced in \cite{R} the notion of the dimension of a triangulated category. Let us recall its definition.
Suppose $\D$ is a triangulated category.  For a full subcategory $\I$ of $\D$, denote by $\pair{\I}$ the smallest full subcategory of $\D$ containing $\I$ and closed under finite direct sums, direct summands and shifts.
For two full subcategories $\I_1, \I_2$ of $\D$, 
 denote $\I_1\diamond\I_2=\pair{\I_1*\I_2}$. Now put $\pair{\I}_0=0$ and define inductively $\pair{\I}_i=\pair{\I}_{i-1}\diamond \pair{\I}$ for $i\geq 1$. Put $\pair{\I}_{\infty}=\bigcup_{i\geq 0}\pair{\I}_i$.  

\begin{defn}
	The dimension $\dim \D$ of $\D$ is defined as the minimal interger $d\geq 0$ such that there is some $E\in \D$ such that $\pair{E}_{d+1}=\D$. If there is no such $E$ then define $\dim \D=\infty$. 
\end{defn}
The following simple fact is enough for us.

\begin{lem}[{\cite{R}}]\label{dense retract}
Let $\D$ be a triangulated category. If there is an exact functor dense up to direct summands $F\colon \D\ra \D'$  between triangulated categories then $\dim \D'\leq \dim \D$. If $\C$ is a full triangulated subcategory of $\D$ and each object in $\D$ is  a direct summand of some object in $\C$ then $\dim \C=\dim \D$.  
In particular, a triangulated category has the same dimension with its idempotent completion.

\end{lem}
%
%

We use the above fact to show that dimension is invariant through equivariantization.
\begin{prop}\label{dim}
	Let $\D$ be a triangulated category satisfying {\bf (Hypothesis)}. We have $\dim \D^G=\dim \D$.
\end{prop}
\begin{proof}
	First suppose that $\D$ is idempotent-complete. Then both the  induction functor and the forgetful functor are exact functors, dense up to direct summands.  By Lemma~\ref{dense retract}, we have $\dim \D^G=\dim \D$. In general, we have \[\dim \D^G=\dim \widehat{\D^G}=\dim \widehat{\D}^G=\dim \widehat{\D}=\dim \D,\] where $\widehat{\A}$ denotes the idempotent-completion of a category $\A$.
\end{proof}

\begin{exm} Let $X$ be a separated scheme of finite type over a perfect field $k$. Suppose $X$ has a right $G$-action by $k$-automorphisms and  $|G|$ is invertible in $k$. Denote by $[X/G]$ the associated quotient stack and by $\coh[X/G]$ the category of coherent sheaves on $[X/G]$.  Then we have 
	\[\begin{aligned} \dim \D^b(\coh[X/G]) &  =\dim \D^b(\coh^G(X)) && \text{since $\coh[X/G]\simeq \coh^G(X)$}\\
			& =\dim \D^b(\coh(X))^G &&\text{since $\D^b(\coh^G(X))\simeq \D^b(\coh(X))^G$}\\
			& =\dim \D^b(\coh(X))&& \text{by Proposition~\ref{dim}}\\
			& <\infty && \text{by \cite[Theorem 7.38]{R}.}
	\end{aligned}\]

\end{exm}

Let $k$ be a field. A triangulated $k$-category $\D$ is said to be \emph{of finite type} if \[\oplus_{i\in\ZZ}\hom^i_\D(X,Y)\] is finite dimensional for any $X, Y\in \D$. Let us recall the definition of saturatedness of a triangulated category in the sense of Bondal and Kapranov \cite{BK}. 
\begin{defn}
	A triangulated $k$-category $\D$ of finite type is called \emph{left} (resp. \emph{right}) \emph{saturated} if any cohomological functor \[H\colon \D\lra \mod k \quad (\text{resp.}\,\, H\colon \D^{\op}\lra \mod k)\] of finite type (i.e., such that $\oplus_{i\in\ZZ} H(A[i])$ is  finite dimensional for any $A\in \D$) is representable. $\D$ is called \emph{saturated} if it is both left and right saturated.
\end{defn}
We show that saturatedness is inherited through equivariantization.
\begin{prop}\label{saturated}
	Suppose $\D$ is an idempotent-complete triangulated category of finite type satisfying {\bf (Hypothesis)}.   If $\D$ is left saturated (resp. right saturated,  saturated) then so is $\D^G$. The converse also holds if $G$ is solvable.
	
\end{prop}
\begin{proof}
	We show that if $\D$ is left saturated then so is $\D^G$. Similar arguments allow one to conclude that if $\D$ is right saturated then so is $\D^G$ and thus $\D^G$ is saturated if $\D$ is. Let $H\colon \D^G\ra \mod k$ be a cohomological functor of finite type. Then $H\circ \ind\colon \D\ra \mod k$ is a cohomological functor of finite type. Now that $\D$ is left saturated, we have a natural isomorphism $H\circ \ind\cong \hom_{\D}(A, -)$ for an object $A\in \D$. So \[H\circ \ind \circ \omega\cong \hom_\D(A, \omega(-))\cong \hom_{\D^G}(\ind(A), -).\] Recall that the composition $\id \ra \ind \circ\omega\ra \id $ of adjunctions  is $|G|\cdot \id$. Thus $H$ is a direct summand of $H\circ \ind \circ \omega$. Since $H\circ \ind\circ \omega$ is representable and since $\D^G$ is idempotent-complete, $H$ is also representable. This shows that $\D^G$ is left saturated. When  $G$ is solvable, by combining Propsition~\ref{A^G A} and the conclusion that we just obtained, we see that  $\D$ is left saturated (resp. right saturated,  saturated) iff so is $\D^G$.
\end{proof}

\end{document}